\documentclass[article]{amsart}

\usepackage{CJK,CJKnumb}
\usepackage{color}
\usepackage{indentfirst}
\usepackage{latexsym,bm}
\usepackage{amsmath,amssymb}
\usepackage{graphicx}
\usepackage{bbm}
\usepackage{amsthm}
\usepackage[ansinew]{inputenc}
\usepackage{array}
\usepackage{amsxtra}
\usepackage{amstext}
\usepackage{latexsym}
\usepackage{amsfonts}
\usepackage{bm}
\usepackage[pdftex,pdfstartview=FitH,bookmarksnumbered,bookmarksopen,bookmarksopenlevel=2,CJKbookmarks]{hyperref}

\newtheorem{theorem}{Theorem}[section]

\newtheorem{remark}[theorem]{Remark}

\newtheorem{proposition}[theorem]{Proposition}

\newtheorem{lemma}[theorem]{Lemma}

 \numberwithin{equation}{section}

\begin{document}

\title
{Hyperbolicity of the time-like extremal surfaces in Minkowski spaces
}

\author{
Xianglong Duan}
\date\today



\address{
CNRS UMR 7640 \\ 
Ecole Polytechnique   \\ 
Palaiseau\\
France}

\email{xianglong.duan@polytechnique.edu}

\begin{abstract}

In this paper, it is established, in the case of graphs, that time-like extremal surfaces of dimension $1+n$
in the Minkowski space of dimension $1+n+m$ can be described by a symmetric
hyperbolic system of PDEs with the very simple structure (reminiscent of the inviscid
Burgers equation)
$$
 \partial_t W + \sum_{j=1}^n A_j(W)\partial_{x_j} W =0,\;\;\;W:\;(t,x)\in\mathbb{R}^{1+n}
\rightarrow W(t,x)\in\mathbb{R}^{n+m+\binom{m+n}{n}},
$$
where each $A_j(W)$ is just a $\big(n+m+\binom{m+n}{n}\big)\times\big(n+m+\binom{m+n}{n}\big)$ symmetric matrix depending
linearly on $W$.

\end{abstract}
\maketitle


\section*{Introduction}



In the $(1+n+m)-$dimensional Minkowski space $\mathbb{R}^{1+(n+m)}$, we consider a time-like
$(1+n)-$dimensional surface (called $n-$brane in String Theory \cite{Po}), namely,
$$
(t,x)\in\overline\Omega \subset{\mathbb{R}\times\mathbb{R}^n}\rightarrow X(t,x)
=(X^0(t,x),\ldots,X^{n+m}(t,x))\in
\mathbb{R}^{1+(n+m)},
$$
where $\Omega$ is a bounded open set.
This surface is called an extremal surface if $X$
is a critical point,
with respect to compactly supported perturbations in the open set $\Omega$,
of the following area functional (which is the Nambu-Goto action in the case $n=1$)
$$
-\iint_\Omega  \sqrt{-\det (G_{\mu\nu})\;\;},
\;\;\;\;\;
G_{\mu\nu}
=\eta_{MN}\partial_{\mu}X^{M}\partial_{\nu}X^{N}\;,
$$
where $M,N=0,1,\ldots,n+m$,\;\; $\mu,\nu=0,1,\ldots,n$, and $\eta=(-1,1,\ldots,1)$ denotes the Minkowski metric, while $G$ is the induced metric on the $(1+n)-$surface by $\eta$.
Here $\partial_0=\partial_t$ and we use the convention that the sum is taken for repeated indices.


By variational principles, the Euler-Lagrange equations gives the well-known
equations of extremal surfaces,
\begin{equation}\label{eq:NG}
\partial_{\mu}\left(\sqrt{-G}G^{\mu\nu}
\partial_{\nu}X^{M}\right)=0,\qquad M=0,1,\ldots,n+m,
\end{equation}
where $G^{\mu\nu}$ is the inverse of $G_{\mu\nu}$ and $G=\det (G_{\mu\nu})$.
In this paper, we limit ourself to
the case of extremal surfaces that are graphs of the form:
\begin{equation}\label{graph}
X^0=t,\;\;X^{i}=x^i,\;\;i=1,\ldots,n,\;\;X^{n+\alpha}=X^{n+\alpha}(t,x),\;\;\alpha=1,\ldots,m
\end{equation}

The main purpose of this paper is to prove:
\begin{theorem}
\label{main}
In the case of a graph as (\ref{graph})
the equations of extremal surfaces (\ref{eq:NG})
can be translated into a
first order symmetric hyperbolic
system of
PDEs,
which admits the
very simple form
\begin{equation}\label{burgers}
\partial_t W + \sum_{j=1}^n A_j(W)\partial_{x_j} W =0,\;\;\;W:(t,x)\in\mathbb{R}^{1+n}
\rightarrow W(t,x)\in\mathbb{R}^{n+m+\binom{m+n}{n}},
\end{equation}
where
each $A_j(W)$ is just a $(n+m+\binom{m+n}{n})\times(n+m+\binom{m+n}{n})$ symmetric
matrix depending linearly on $W$.
Accordingly, this system is automatically well-posed, locally in time, in the Sobolev space
$W^{s,2}$ as soon as $s>n/2+1$.

\end{theorem}

The structure of (\ref{burgers}) is reminiscent of the celebrated prototype of
all nonlinear hyperbolic PDEs, the so-called inviscid Burgers equation
$\partial_t u+u\partial_x u=0$, where $u$ and $x$ are both just valued in $\mathbb{R}$,
with the simplest possible nonlinearity.
Of course, to get such a simple structure,
the relation to be found between $X$ (valued in $\mathbb{R}^{1+n+m}$)
and $W$ (valued in $\mathbb{R}^{n+m+\binom{m+n}{n}}$) must be
quite involved. Actually, it will be shown more precisely that the case of extremal surfaces corresponds
to a special subset of solutions of (\ref{burgers}) for which $W$ lives in a very special
algebraic sub-manifold of $\mathbb{R}^{n+m+\binom{m+n}{n}}$, which is preserved by the
dynamics of (\ref{burgers}).

To establish Theorem \ref{main}, the strategy of proof follows the concept of system of conservation
laws with ``polyconvex'' entropy in the sense of Dafermos \cite{Da}. The first step is to lift
the original system of conservation laws to a (much) larger one which enjoys a convex
entropy rather than a polyconvex one. This strategy has been
successfully applied in many situations, such as nonlinear Elastodynamis \cite{DST,Qin},
nonlinear Electromagnetism \cite{Br,BY,Se}, just to quote few examples.
In our case, the calculations will crucially start with the classical Cauchy-Binet formula.

Finally, at the end of the paper, following the ideas recently introduced in \cite{BD},
we will make a connection between our result and the theory of mean-curvature flows
in the Euclidean space, in any dimension and co-dimension.

\subsection*{Acknowledgements}
The author is very grateful to his thesis advisor, Yann Brenier, for introducing the polyconvex system to him and pointing out the possibility of augmenting this system as a hyperbolic system of conservation laws, in the spirit of \cite{Br}.

\section{Extremal surface equations for a graph}

Let us first write equations (\ref{eq:NG}) in the case of a graph such as (\ref{graph}).
We denote
$$V_{\alpha}=\partial_{t}X^{n+\alpha},\;\;F_{\alpha i}=\partial_{i}X^{n+\alpha},\;\;\alpha=1,\ldots,m,\;\;i=1,\ldots,n.$$
Then the induced metric tensor $G_{\mu\nu}$ can be written as
$$
(G_{\mu\nu})=\left(
                \begin{array}{cc}
                  -1+|V|^2 & V^TF \\
                 F^TV   & I_n + F^TF \\
                \end{array}
              \right).
$$
We can easily get that
$$G=\det (G_{\mu\nu})=-\det(I_n + F^TF)\big(1-|V|^2 + V^TF(I_n + F^TF)^{-1}F^TV \big).$$
So, in the case of graph, the extremal surface can be solved by varying the follow Lagrangian of the vector $V$ and matrix $F$,
$$
\iint L(V,F),\;\;\;L(V,F)=-\sqrt{-G},
$$
under the constraints
$$\partial_t F_{\alpha i}=\partial_i V_{\alpha},\;\;\partial_i F_{\alpha j}=\partial_j F_{\alpha i},\;\;\alpha=1,\ldots,m,\;\;i,j=1,\ldots,n.$$
The resulting system combines the above constraints and
$$\partial_t \left(\frac{\partial L(V,F)}{\partial V_{\alpha}}\right) + \partial_i \left(\frac{\partial L(V,F)}{\partial F_{\alpha i}}\right)=0.$$
Now let us denote
$$D_{\alpha}=\frac{\partial L(V,F)}{\partial V_{\alpha}}=\frac{\sqrt{\det(I_n + F^TF)}(I_m+FF^T)^{-1}_{\alpha\beta}V_{\beta}}{\sqrt{1-V^T(I_m+FF^T)^{-1}V}}$$
and the energy density $h$ by
$$h(D,F)=\sup_{V}D\cdot V-L(V,F)=\sqrt{\det(I_n + F^TF)+D^T(I_m+FF^T)D}.$$
We have
$$V_{\alpha}=\frac{\partial h(D,F)}{\partial D_{\alpha}}=\frac{(I_m+FF^T)_{\alpha\beta}D_{\beta}}{h}.$$
So, the extremal surface should solve the following system for a matrix valued function $F=(F_{\alpha i})_{m\times n}$ and a vector valued function $D=(D_{\alpha})_{\alpha=1,2,\ldots,m}$,
\begin{equation}\label{eq:so1}
\partial_tF_{\alpha i} + \partial_i\left(\frac{D_{\alpha}+F_{\alpha j}P_j}{h}\right)=0,
\end{equation}
\begin{equation}\label{eq:so2}
\partial_tD_{\alpha} + \partial_i\left(\frac{D_{\alpha}P_i+\xi'(F)_{\alpha i}}{h}\right)=0,
\end{equation}
\begin{equation}\label{eq:so3}
\partial_j F_{\alpha i}=\partial_{i}F_{\alpha j},\quad 1\leq i,j\leq n,\;1\leq\alpha\leq m,
\end{equation}
where
\begin{equation}
P_i=F_{\alpha i}D_{\alpha},\;\;\;h=\sqrt{D^2+P^2+\xi(F)},\;\;\; 1\leq i,j\leq n,\;1\leq\alpha\leq m,
\end{equation}
\begin{equation}
\xi(F)=\det\big(I_n+F^TF\big),\;\;\;\xi'(F)_{\alpha i}=\frac{1}{2}\frac{\partial\xi(F)}{\partial F_{\alpha i}}=\xi(F)(I_n+F^TF)^{-1}_{ij}F_{\alpha j}.
\end{equation}
In fact, we can get the above equations directly from \eqref{eq:NG}. Interested readers can refer to Appendix A for the details. Moreover, we can find that there are other conservation laws for the energy density $h$ and vector $P$ as defined in the above equations, namely, (see Appendix B)
\begin{equation}\label{eq:so4}
\partial_th+\nabla\cdot P=0,
\end{equation}
\begin{equation}\label{eq:so5}
\partial_tP_i + \partial_j\left(\frac{P_i P_j}{h}-\frac{\xi(F)(I_n+F^T F)^{-1}_{ij}}{h}\right)=0.
\end{equation}
Now, let's take $h$ and $P$ as independent variables, then we can find that the system \eqref{eq:so1},\eqref{eq:so2},\eqref{eq:so3},\eqref{eq:so4},\eqref{eq:so5} admits an additional conservation law for
$$S=\frac{D^2+P^2+\xi(F)}{2h},$$
namely,
\begin{equation}
\partial_t S + \nabla\cdot\left(\frac{SP}{h}\right) = \partial_i\left[\frac{\xi(F)(I_n+F^T F)^{-1}_{ij}(P_j-F_{\alpha j}D_{\alpha})}{h^2}\right]
\end{equation}



\section{Lifting of the system}


\subsection{The minors of the matrix $F$}

In previous part, $S$ is generally not a convex function of $(h,D,P,F)$, but a polyconvex function of $F$, which means that $S$ can be written as convex functions of the minors of $F$. Now we denote $r=\min\{m,n\}$. For $1\leq k\leq r$, and any ordered sequences $1\leq \alpha_1<\alpha_2<\ldots<\alpha_k\leq m$ and $1\leq i_1<i_2<\ldots<i_k\leq n$, let $A=\{\alpha_1,\alpha_2,\ldots,\alpha_k\}$, $I=\{i_1,i_2,\ldots,i_k\}$, then the minor of $F$ with respect to the rows $\alpha_1,\alpha_2,\ldots,\alpha_k$ and columns $i_1,i_2,\ldots,i_k$ is defined as
$$[F]_{A,I}=\det\Big((F_{\alpha_p i_q})_{p,q=1,\ldots,k}\Big)$$
For the minors $[F]_{A,I}$, let us first introduce the generalized Cauchy-Binet formula which is very convenient for us to compute the minors of the product of two matrices.
\begin{lemma}
{\rm\bf(Cauchy-Binet formula)} Suppose $M$ is a $m\times l$ matrix, $N$ is a $l\times n$ matrix, $I$ is a subset of $\{1,2,\ldots,m\}$ with $k(\leq l)$ elements and $J$ is a subset of $\{1,2,\ldots,n\}$ with $k$ elements, then
\begin{equation}
[MN]_{I,J}=\sum_{K\subseteq\{1,2,\ldots,l\} \atop |K|=k}[M]_{I,K}[N]_{K,J}
\end{equation}
\end{lemma}
Now let us look at $\xi(F)=\det\big(I_n+F^T F\big)$, we can show that it is a convex function for the minors $[F]_{A,I}$. In fact, we have,
$$\xi(F)=\det\big(I_n+F^T F\big)=1 + \sum_{k=1}^{n}\sum_{I\subseteq\{1,2,\ldots,n\} \atop|I|=k}[F^T F]_{I,I}$$
(by the Cauchy-Binet formula)
$$= 1+ \sum_{k=1}^{r}\sum_{|I|,|A|=k}[F^T]_{I,A}[F]_{A,I}$$
So we have
\begin{equation}
\xi(F)=1+\sum_{k=1}^{r}\sum_{|A|,|I|=k}[F]_{A,I}^2
\end{equation}
The above equality tells us that $\xi(F)$ is a polyconvex function of $F$. By introducing all the minors of $F$ as independent variables, the energy $S$ becomes a strictly convex function of $h,D,P,[F]_{A,I}$. Now we will see that the system can be augmented as a system of conservation laws of $h,D,P,[F]_{A,I}$.


\subsection{Conservation laws for the minors $[F]_{A,I}$}

First, we will see that $[F]_{A,I}$ satisfy similar equations as \eqref{eq:so3}. For simplicity, we denote $[F]_{A,I}=1$ if $A=I=\emptyset$.

\begin{proposition}
Suppose $F$ satisfy \eqref{eq:so3}, then for any $2\leq k\leq r+1$,  $A'=\{1\leq\alpha_1 < \alpha_2<\ldots< \alpha_{k-1}\leq m\}$, $I=\{1\leq i_1 <i_2<\ldots<i_k\leq n \}$, we have
\begin{equation}
\sum_{q=1}^k(-1)^q\partial_{i_q}\Big([F]_{A',I\setminus\{i_q\}}\Big)=0
\end{equation}
\end{proposition}

\begin{proof}
This can be showed quite directly, for the left hand side, we have
\begin{equation*}
\begin{array}{r@{}l}
{\rm Left}\;&\;\displaystyle{ = \sum_{q=1}^k \sum_{l<q \atop 1\leq p\leq k-1} (-1)^{l+p+q}[F]_{A\setminus\{\alpha_p\},I\setminus\{i_l,i_q\}}\partial_{i_q}
F_{\alpha_p i_l}   } \\
&\;\displaystyle{ \;\;\;+ \sum_{q=1}^k \sum_{l>q \atop 1\leq p\leq k-1} (-1)^{l-1+p+q}[F]_{A\setminus\{\alpha_p\},I\setminus\{i_l,i_q\}}\partial_{i_q}
F_{\alpha_p i_l}}\\
&\;\displaystyle{ = \sum_{1\leq l<q\leq k \atop 1\leq p\leq k-1}(-1)^{l+p+q}[F]_{A\setminus\{\alpha_p\},I\setminus\{i_l,i_q\}}\Big(\partial_{i_q}
F_{\alpha_p i_l} - \partial_{i_l}
F_{\alpha_p i_p}\Big)  }\\
&\;= 0
\end{array}
\end{equation*}

\end{proof}

With the above proposition, we can get the conservation laws for $[F]_{A,I}$. For $A=\{1\leq\alpha_1 < \alpha_2<\ldots< \alpha_{k}\leq m\}$, $I=\{1\leq i_1 <i_2<\ldots<i_k\leq n \}$, $1\leq k\leq r$, we have
\begin{equation}
\begin{array}{r@{}l}
\partial_t \big([F]_{A,I}\big)\;&\;\displaystyle{ = \sum_{p,q=1}^{k}(-1)^{p+q}[F]_{A\setminus\{\alpha_p\},I\setminus\{i_q\} } \partial_t F_{\alpha_p i_q} } \\
&\;\displaystyle{ = - \sum_{p,q=1}^{k}(-1)^{p+q} [F]_{A\setminus\{\alpha_p\},I\setminus\{i_q\}}\partial_{i_q} \left(\frac{D_{\alpha_p} + F_{\alpha_p j}P_{j}}{h} \right)  }\\
&\;\displaystyle{ = - \sum_{p,q=1}^{k}(-1)^{p+q} \partial_{i_q} \left[\frac{[F]_{A\setminus\{\alpha_p\},I\setminus\{i_q\}}\big(D_{\alpha_p} + F_{\alpha_p j}P_{j}\big)}{h} \right]  }
\end{array}
\end{equation}


\subsection{The augmented system}

Now let us consider the energy density $h$, the vector field $P$ and the minors $[F]_{A,I}$ as independent variables. The original system \eqref{eq:so1}-\eqref{eq:so3} can be augmented to the following system of conservation laws. More precisely, for $h>0$, $D=(D_{\alpha})_{\alpha=1,2,\ldots,m}$, $P=(P_i)_{i=1,2,\ldots,n}$, $M_{A,J}$ with $A\subseteq\{1,2,\ldots,m\}$, $I\subseteq\{1,2,\ldots,n\}$, $1\leq|A|=|I|\leq r=\min\{m,n\}$, the system are composed of the following equations
\begin{equation}\label{eq:sn1}
\partial_t h + \nabla\cdot P = 0
\end{equation}
\begin{equation}\label{eq:sn2}
\partial_tD_{\alpha} + \partial_i\left(\frac{D_{\alpha}P_i}{h}\right)  + \sum_{A,I,i\atop\alpha\in A,i\in I}(-1)^{O_A(\alpha)+O_I(i)}\partial_i\left(\frac{M_{A,I}M_{A\setminus \{\alpha\},I\setminus\{i\} }}{h}\right)=0
\end{equation}
\begin{multline}\label{eq:sn3}
\partial_t P_i + \sum_{A,I,j\atop j\in I,i\notin I\setminus\{j\}}(-1)^{O_{I}(j)+O_{I\setminus\{j\}}(i)}\partial_j\left(\frac{M_{A,(I\setminus\{j\})\bigcup\{i\}}M_{A,I}}{h}\right)
\\+ \partial_j\left(\frac{P_i P_j}{h}\right) -\partial_i\left(\frac{1+\sum_{A,I}M_{A,I}^2}{h}\right)=0
\end{multline}
\begin{multline}\label{eq:sn4}
\partial_t M_{A,I} + \sum_{i,j\atop i\in I,j\notin I\setminus\{i\}}(-1)^{O_{I\setminus\{i\}}(j)+O_{I}(i)} \partial_{i}\left(\frac{M_{A,(I\setminus\{i\})\bigcup\{j\}}P_j}{h} \right) \\ + \sum_{\alpha,i\atop \alpha\in A, i\in I}(-1)^{O_{A}(\alpha)+O_I(i)} \partial_{i}\left(\frac{M_{A\setminus\{\alpha\},I\setminus\{i\}} D_{\alpha}}{h}\right) =0
\end{multline}
\begin{equation}\label{eq:sn5}
\sum_{i\in I}(-1)^{O_I(i)}\partial_{i}\Big(M_{A',I\setminus\{i\}}\Big)=0,
\quad  2\leq |I|=|A'|+1\leq r+1
\end{equation}
Here $O_A(\alpha)$ represents the number such that $\alpha$ is the $O_A(\alpha)$th smallest element in $A\bigcup\{\alpha\}$. All the sum are taken in the convention that $A\subseteq\{1,\ldots,m\}$, $I\subseteq\{1,\ldots,n\}$, $1\leq\alpha\leq m$, $1\leq i,j\leq n$.

Note that there are many different ways to enlarge the original system since the equations can be written in many different ways in terms of minors. Although our above augmented system looks quite complicated, in the following part, we will show that by extending the system in this way is quite useful. Now, let's first show that the augmented system can be reduced to the original system under the algebraic constraints we abandoned we enlarge the system.
\begin{proposition}
We can recover the original system \eqref{eq:so1}-\eqref{eq:so3} from the augmented system \eqref{eq:sn1}-\eqref{eq:sn5}under the algebraic constrains
$$P_i=F_{\alpha i}D_{\alpha},\;\;\;h=\sqrt{D^2+P^2+\xi(F)},\;\;\;M_{A,I}=[F]_{A,I}$$
\end{proposition}

\begin{proof}

It suffices to show the following three equalities,
\begin{equation}\label{eq:np1}
\xi'(F)_{\alpha i}= \sum_{A,I\atop\alpha\in A,i\in I}(-1)^{O_A(\alpha)+O_I(i)}[F]_{A,I}[F]_{A\setminus \{\alpha\},I\setminus\{i\}}
\end{equation}
\begin{multline}\label{eq:np2}
\xi(F)(I_{n}+F^T F)_{ij}^{-1}= (1+\sum_{A,I}[F]_{A,I}^{2})\delta_{ij} \\ - \sum_{A,I\atop j\in I,i\notin I\setminus\{j\}}(-1)^{O_{I}(j)+O_{I\setminus\{j\}}(i)}[F]_{A,(I\setminus\{j\})\bigcup\{i\}}[F]_{A,I}
\end{multline}
\begin{equation}\label{eq:np3}
\sum_{p=1}^{k}(-1)^{p+q} [F]_{A\setminus\{\alpha_p\},I\setminus\{i_q\}}F_{\alpha_p j}=\begin{cases}(-1)^{O_{(I\setminus\{i_q\})\bigcup\{j\}}(j)+q} [F]_{A,(I\setminus\{i_q\})\bigcup\{j\}} & j\notin I\setminus\{i_q\}\\ 0 & j\in I\setminus\{i_q\} \end{cases}
\end{equation}
\eqref{eq:np3} is obvious because of the Laplace expansion. Now, since
\begin{equation*}
\begin{array}{r@{}l}
\displaystyle{\xi'(F)_{\alpha i} \;} & \displaystyle{= \; \frac{1}{2}\frac{\partial}{\partial F_{\alpha i}}\Big( 1 + \sum_{A,I}[F]^2_{A,I}\Big) \;=\; \sum_{A,I} [F]_{A,I} \frac{\partial}{\partial F_{\alpha i}}\Big([F]_{A,I}\Big) }\\
& \displaystyle{ =\; \sum_{A,I\atop\alpha\in A,i\in I}(-1)^{O_A(\alpha)+O_I(i)}[F]_{A,I}[F]_{A\setminus \{\alpha\},I\setminus\{i\}}  }
\end{array}
\end{equation*}
so \eqref{eq:np1} is true. Let's look at \eqref{eq:np2}. First, we have
\begin{equation*}
\begin{array}{r@{}l}
\displaystyle{\xi(F)\delta_{ij}-\xi(F)(I_n + F^T F)_{ij}^{-1} \;} & \displaystyle{=\;\xi(F)F_{\alpha i}(I_m + F F^T)^{-1}_{\alpha\beta} F_{\beta j}  }\\
& \displaystyle{=\; (-1)^{\alpha+\beta}F_{\alpha i}[I_m + F F^T]_{\{\alpha\}^c,\{\beta\}^c}F_{\beta j} }
\end{array}
\end{equation*}
Because
\begin{equation*}
\begin{array}{r@{}l}
& \;\;\;\displaystyle{[I_m + F F^T]_{\{\alpha\}^c,\{\beta\}^c} \;} \\
=& \displaystyle{ \;\; \sum_{k=0}^{m-1}\sum_{|A'|=k\atop \alpha,\beta\notin A' } (-1)^{O_{A'}(\alpha) + O_{A'} (\beta) } [FF^T]_{(A'\bigcup\{\alpha\})^c,(A'\bigcup\{\beta\})^c} }\\
=& \displaystyle{ \;\; \sum_{k=1}^{m}\sum_{|A|=k\atop \alpha,\beta\in A } (-1)^{O_{A}(\alpha) + O_{A} (\beta) +\alpha +\beta} [FF^T]_{A\setminus\{\alpha\},A\setminus\{\beta\}} }\\
=& \displaystyle{ \;\; \sum_{k=1}^{\min\{m,r+1\}}\sum_{|A|=k,|I'|=k-1\atop \alpha,\beta\in A } (-1)^{O_{A}(\alpha) + O_{A} (\beta) +\alpha +\beta } [F]_{A\setminus\{\alpha\},I'}[F]_{A\setminus\{\beta\},I'}       }
\end{array}
\end{equation*}
then we have
\begin{equation*}
\begin{array}{r@{}l}
& \;\;\;\displaystyle{\xi(F)\delta_{ij}-\xi(F)(I_n + F^T F)_{ij}^{-1} } \\
= &\;\; \displaystyle{ \sum_{k=1}^{\min\{m,r+1\}}\sum_{|A|=k,|I'|=k-1\atop \alpha,\beta\in A } (-1)^{O_{A}(\alpha) + O_{A} (\beta) } [F]_{A\setminus\{\alpha\},I'}[F]_{A\setminus\{\beta\},I'} F_{\alpha i} F_{\beta j}  }\\
= & \displaystyle{ \;\; \sum_{k=1}^{r}\sum_{|A|=k,|I'|=k-1\atop i,j\notin I' } (-1)^{ O_{I'}(i)+O_{I'}(j)} [F]_{A,I'\bigcup\{i\}}[F]_{A,I'\bigcup\{j\}}  }\\
= & \displaystyle{\sum_{A,I\atop j\in I,i\notin I\setminus\{j\}}(-1)^{O_{I}(j)+O_{I\setminus\{j\}}(i)}[F]_{A,(I\setminus\{j\})\bigcup\{i\}}[F]_{A,I} }
\end{array}
\end{equation*}

\end{proof}

Now we can show that the augmented system have a convex entropy.

\begin{proposition}
The system \eqref{eq:sn1}-\eqref{eq:sn5} satisfies an additional conservation law for
$$S(h,D,P,M_{A,I})=\frac{1+ D^2 + P^2 + \sum_{A,I}M_{A,I}^2}{2h}$$
More precisely, we have
\begin{multline}
\partial_t S + \nabla\cdot\left(\frac{SP}{h}\right) + \sum_{A,I,i\atop\alpha\in A,i\in I}(-1)^{O_A(\alpha)+O_I(i)}\partial_i\left(\frac{D_{\alpha}M_{A\setminus \{\alpha\},I\setminus\{i\}}M_{A,I} }{h^2}\right) \\ + \sum_{A,I,j\atop j\in I,i\notin I\setminus\{j\}}(-1)^{O_{I}(j)+O_{I\setminus\{j\}}(i)}\partial_j\left(\frac{P_i M_{A,(I\setminus\{j\})\bigcup\{i\}}M_{A,I}}{h^2}\right) -\partial_j\left(\frac{P_j(1+M_{A,I}^2)}{h^2}\right)=0
\end{multline}

\end{proposition}

We leave the proof in Appendix C.

\begin{remark}
There are many possible ways to augment the original system because of the different ways to write a function of minors. To find the write way to express the equation \eqref{eq:so2} and \eqref{eq:so5} such that it has a convex entropy $S$ is somehow a little technical.
\end{remark}


\section{Properties of the augmented system}


\subsection{Propagation speeds and characteristic fields}

Let's look at the special case $n=1$, where our extremal surface is just a relativistic string. In this case, the augmented system coincides with the system of $h,P,D,F$, where the $P$ is a scalar function and $F=(F_{\alpha})_{\alpha=1,\ldots,m}$ becomes a vector. More precisely, the equations in the case $n=1$ are,
$$
\partial_th+\partial_x P=0,\;\;\;\partial_tF_{\alpha i} + \partial_x\left(\frac{D_{\alpha}+F_{\alpha}P}{h}\right)=0,
$$
$$
\partial_tP + \partial_x\left(\frac{P^2-1}{h}\right)=0,\;\;\;\partial_tD_{\alpha} + \partial_x\left(\frac{D_{\alpha}P+F_{\alpha}}{h}\right)=0.
$$
Let us denote $U=(h,P,D_{\alpha},F_{\alpha})$ then, the system can be written as
$$
\partial_t U + A(U)\partial_x U=0,
$$
where
$$
A(U)=\frac{1}{h}\left(\begin{array}{cccc}
0 & h & 0 & 0\\
\frac{1-P^2}{h} & 2P & 0 & 0\\
-\frac{PD+F}{h} & D & PI_m & I_m\\
-\frac{D+PF}{h} & F & I_m & PI_m\\
\end{array}\right)
$$
We can find that, the propagation speeds are
$$\lambda_{+}=\frac{P+1}{h},\;\;\;\lambda_{-}=\frac{P-1}{h}$$
with each of them having multiplicity $m+1$. The characteristic field for $\lambda_{+}$ is composed of
$$v_{+}^0=(h,P+1,D,F),\;\;v_{+}^i=(0,0,e_i,e_i),\;\;i=1,\ldots,m.$$
Here $e_i$ is the base of $\mathbb{R}^m$. The characteristic field for $\lambda_{-}$ is composed of
$$v_{-}^0=(h,P-1,D,F),\;\;v_{-}^i=(0,0,e_i,-e_i),\;\;i=1,\ldots,m.$$
We can easily check that
$$\frac{\partial \lambda_{+}(U)}{\partial U}\cdot v_{+}^i(U)=0,\;\;\;\frac{\partial \lambda_{-}(U)}{\partial U}\cdot v_{-}^i(U)=0,\;\;\;i=0,1,\ldots,m.$$
So the augmented system is linearly degenerate in the sense of the theory of hyperbolic conservation
laws \cite{Da}.


\subsection{Non-conservative form}

Now let's look at the non-conservative form of the augmented system \eqref{eq:sn1}-\eqref{eq:sn5}. We denote
$$\tau=\frac{1}{h},\;\;\; d=\frac{D}{h},\;\;\; v=\frac{P}{h},\;\;\;m_{A,I}=\frac{M_{A,I}}{h}.$$
For simplicity, we denote $m_{A,I}=\tau$ if $A=I=\emptyset$. We have the following proposition.
\begin{proposition}
Suppose $(h,D,P,M_{A,I})$ is a smooth solution of \eqref{eq:sn1}-\eqref{eq:sn5}, then $(\tau,d,v,m_{A,I})$ is the solution of the following symmetric hyperbolic system,
\begin{equation}\label{eq:snc1}
\partial_t \tau + v_j\partial_j\tau-\tau\partial_jv_j = 0
\end{equation}
\begin{equation}\label{eq:snc2}
\partial_td_{\alpha} + v_i\partial_id_{\alpha}  + \sum_{A,I,i}\mathbbm{1}_{\{\alpha\in A,i\in I\}}(-1)^{O_A(\alpha)+O_I(i)}m_{A\setminus \{\alpha\},I\setminus\{i\} }\partial_im_{A,I}=0
\end{equation}
\begin{multline}\label{eq:snc3}
\partial_t v_i +
\sum_{A,I,j}\mathbbm{1}_{\{j\in I,i\notin I\setminus\{j\}\}}(-1)^{O_{I}(j)+O_{I\setminus\{j\}}(i)}m_{A,(I\setminus\{j\})\bigcup\{i\}}\partial_jm_{A,I}
\\ - \sum_{A,I}m_{A,I}\partial_{i}m_{A,I} + v_j\partial_j v_i - \tau\partial_i\tau=0
\end{multline}
\begin{multline}\label{eq:snc4}
\partial_t m_{A,I} + v_j\partial_jm_{A,I} + \sum_{i,j}\mathbbm{1}_{\{i\in I,j\notin I\setminus\{i\}\}}(-1)^{O_{I\setminus\{i\}}(j)+O_{I}(i)} m_{A,(I\setminus\{i\})\bigcup\{j\}}\partial_{i}v_j  \\ - m_{A,I}\partial_j v_j +  \sum_{\alpha,i}\mathbbm{1}_{\{\alpha\in A,i\in I\}}(-1)^{O_A(\alpha)+O_I(i)} m_{A\setminus\{\alpha\},I\setminus\{i\}} \partial_{i}d_{\alpha} =0
\end{multline}
\end{proposition}

We can prove the above proposition by just using \eqref{eq:sn5}. It is easy to verify that this system is symmetric. If we set $W=(\tau,d_{\alpha},v_i,m_{A,I})
\in\mathbb{R}^{n+m+\binom{m+n}{n}}$, then the equations can be written as
$$\partial_t W + \sum_{j}A_j(W)\partial_j W =0,$$
where $A_j(W)$ is a symmetric matrix, and more surprisingly, it is a linear function of $W$.
This is exactly the form (\ref{burgers}) announced in the introduction.
Notice that this system does not require any restriction on the range of $W$! In particular the
variable $\tau$ may admit positive, negative or null values. This is a very remarkable situation,
if we compare with more classical nonlinear hyperbolic systems, such as the Euler equations of
gas dynamics (where typically $\tau$ should admit only positive values).
\\
\\
Now let us prove that the two system are equivalent when initial data satisfies \eqref{eq:sn5}

\begin{proposition}
Suppose the initial data for \eqref{eq:snc1}-\eqref{eq:snc4} satisfies \eqref{eq:sn5}, i.e.,
$$\sum_{i\in I}(-1)^{O_I(i)}\partial_{i}\Big(\tau^{-1}m_{A',I\setminus\{i\}}\Big)=0,
\quad  2\leq |I|=|A'|+1\leq r+1$$
then the corresponding smooth solutions satisfy \eqref{eq:sn1}-\eqref{eq:sn5}.
\end{proposition}

\begin{proof}
We only need to proof that the smooth solutions always satisfy \eqref{eq:sn5} provided that initial data satisfies it. For $2\leq |I|=|A'|+1\leq r+1$, let us denote $$\sigma_{A',I}=\sum_{i\in I}(-1)^{O_I(i)}\partial_{i}\Big(\tau^{-1}m_{A',I\setminus\{i\}}\Big)=0$$
then by \eqref{eq:snc1},\eqref{eq:snc4}, we have
\begin{equation*}
\begin{array}{r@{}l}
{\partial_t\sigma_{A',I} }\;&\displaystyle{ = \;  \sum_{i\in I}(-1)^{O_I(i)}\partial_{i}\Big(\tau^{-1}\partial_t m_{A',I\setminus\{i\}} -\tau^{-2} m_{A',I\setminus\{i\}}\partial_t\tau\Big) } \\
&\;\displaystyle{ = \; -v_j\partial_j\sigma_{A',I}  + \sum_{\alpha,i\atop \alpha\in A',i\in I}(-1)^{O_{A'}(\alpha)+O_I(i)} \sigma_{A'\setminus\{\alpha\},I\setminus\{i\}}\partial_{i}d_{\alpha} }\\
&\;\displaystyle{  \;\;\; -\sum_{i,j\atop i\in I,j\notin I\setminus\{i\}}(-1)^{O_{I\setminus\{i\}}(j)+O_I(i)}\sigma_{A',(I\setminus\{i\})\bigcup\{j\}}\partial_{i}v_j }
\end{array}
\end{equation*}
Then we have the following estimate,
$$\partial_{t}\sum_{A',I}\int\sigma_{A',I}^2 \leq C(\|\nabla v\|_{\infty},\|\nabla d\|_{\infty})\Big(\sum_{A',I}\int\sigma_{A',I}^2\Big)$$
Since the initial data $\sigma_{A',I}(0)=0$, then by Gronwall's lemma, we have $\sigma_{A',I}\equiv0$. With these equalities, it is easy to prove the statement just by doing the reverse computation as in the previous proposition.

\end{proof}

Now let us look at the connection with the original system. It is obvious that the non-conservative form of the augmented system is symmetric, thus, the initial value problem is at least locally well-posed. But for the original system, this kind of property is not obvious. However, we can show that, the augmented system is equivalent to the original system if the initial value satisfy the following constraints
\begin{equation}
P_i=F_{\alpha i}D_{\alpha},\;\;\;h=\sqrt{D^2+P^2+\xi(F)},\;\;\;M_{A,I}=[F]_{A,I}
\end{equation}
or, in the non-conservative form,
\begin{equation}\label{eq:non-cons}
\tau v_i=m_{\alpha i}d_{\alpha},\;\;\;1=d_{\alpha}^2+v_i^2+\tau^2+m_{A,I}^2,\;\;\; m_{A,I}=\tau[F]_{A,I}
\end{equation}
Now let us denote
$$
\lambda=\frac{1}{2}(\tau^2+v_i^2+d_{\alpha}^2+m_{A,I}^2-1),\;\;\;
\omega_i=\tau v_i-m_{\alpha i}d_{\alpha}
$$
$$
\varphi_{A,I}^{\alpha}=\sum_{i\in I}(-1)^{O_A(\alpha)+O_I(i)}m_{A,I\setminus\{i\}}m_{\alpha i} - \mathbbm{1}_{\{\alpha\notin A\}}\tau m_{A\bigcup\{\alpha\},I}
$$
$$
\psi_{A,I}^{i}=\sum_{\alpha\in A}(-1)^{O_A(\alpha)+O_I(i)}m_{A\setminus\{\alpha\},I}m_{\alpha i} - \mathbbm{1}_{\{i\notin I\}}\tau m_{A,I\bigcup\{i\}}
$$
It is obvious that $(\tau,v_i,d_{\alpha},m_{A,I})$ satisfy the above constraints \eqref{eq:non-cons} if and only if $\lambda$, $\omega_i$, $\varphi_{A,I}^{\alpha}$,$\psi_{A,I}^{i}$ vanish for all possible choice of $A,I,\alpha,i$. Furthermore, we can show that the algebraic constraints \eqref{eq:non-cons} are preserved by the non-conservative system \eqref{eq:snc1}-\eqref{eq:snc4}. First, we have the following lemma.

\begin{lemma}
If $(\tau,v_i,d_{\alpha},m_{A,I})$ solves the non-conservative system \eqref{eq:snc1}-\eqref{eq:snc4}, then
$\lambda$, $\omega_i$, $\varphi_{A,I}^{\alpha}$,$\psi_{A,I}^{i}$ as defined above satisfy the following equalities,
\begin{equation}\label{eq:cons-1}
\partial_t \omega_i=\omega_i\partial_j v_j-\omega_j\partial_i v_j-v_j\partial_j\omega_i +\tau\partial_i\lambda +\sum_{A,I,j}\mathbbm{1}_{\{j\in I\}}(-1)^{O_{I}(j)+O_{I\setminus\{j\}}(i)}\partial_j m_{A,I}\psi^{i}_{A,I\setminus\{j\}}
\end{equation}
\begin{multline}\label{eq:cons-2}
\partial_t\lambda=-v_j\partial_j\lambda + \tau\partial_i\omega_i +\sum_{A,|I'|\geq 2,i,j}\mathbbm{1}_{\{i,j\in I'\}}(-1)^{O_{I'}(i)+O_{I'}(j)}m_{A,I'\setminus\{j\}}\partial_j\left(\frac{m_{A,I'\setminus\{i\}}\omega_i}{\tau}\right) \\ + \sum_{A',|I|\geq 2,\alpha,i}\mathbbm{1}_{\{i\in I\}}(-1)^{O_{A'}(\alpha)+O_{I}(i)}m_{A',I\setminus\{i\}}\partial_i\left(\frac{
\varphi_{A',I}^{\alpha}d_{\alpha}}{\tau}\right)
\end{multline}
\begin{multline}\label{eq:cons-3}
\partial_t \varphi^{\alpha}_{A,I} =2\varphi^{\alpha}_{A,I}\partial_j v_j-v_j\partial_j\varphi^{\alpha}_{A,I} -\sum_{j,k}\mathbbm{1}_{\{j\in I,k\notin I\setminus\{j\}\}}(-1)^{O_{I\setminus\{j\}}(k)+O_{I}(j)} \varphi_{A,(I\setminus\{j\})\bigcup\{k\}}^{\alpha}\partial_{j}v_k \\ +\sum_{\beta,j}\mathbbm{1}_{\{\beta\in A,j\in I\}}(-1)^{O_{A}(\alpha)+O_{A}(\beta)+O_{A\setminus\{\beta\}}(\alpha)+O_I(j)}\varphi_{A\setminus\{\beta\},I\setminus\{j\}}^{\alpha} \partial_{j}d_{\beta}
\end{multline}
\begin{multline}\label{eq:cons-4}
\partial_t\psi_{A,I}^{i}=2\psi_{A,I}^{i}\partial_j v_j-v_j\partial_j\psi_{A,I}^{i} -\sum_{k}(-1)^{O_I(i)+O_I(k)}\psi_{A,I}^{k}\partial_iv_k \\ + \sum_{\beta\in A,j\in I}(-1)^{O_A(\beta) + O_I(j) +O_I(i)+O_{I\setminus\{j\}}(i)}\psi_{A\setminus\{\beta\},I\setminus\{j\}}^{i}\partial_jd_{\beta} \\ -\sum_{j,k}\mathbbm{1}_{\{j\in I,k\notin I\setminus\{j\}\}}(-1)^{O_I(i)+O_{I}(j)+O_{I\setminus\{j\}}(k)+O_{(I\setminus\{j\})\bigcup\{k\}}(i)} \psi_{A,(I\setminus\{j\})\bigcup\{k\}}^{i}\partial_{j}v_k
\end{multline}

\end{lemma}
The proof of this lemma requires very lengthy and tedious computation. Interesting readers can refer to appendix D for the details of the proof. By the above lemma, we can show that the algebraic constraints are preserved. We summarise our result in the following proposition.

\begin{proposition}
Supposed $(\tau,v_i,d_{\alpha},m_{A,I})$ is a solution to the non-conservative equations \eqref{eq:snc1}-\eqref{eq:snc4} and the initial data satisfies the constraints
$$\tau v_i=m_{\alpha i}d_{\alpha},\;\;\;1=d_{\alpha}^2+v_i^2+\tau^2+m_{A,I}^2,\;\;\; m_{A,I}=\tau[F]_{A,I}$$
where $F_{\alpha i}=\tau^{-1}m_{\alpha i}$, then the above constraints are always satisfied.

\end{proposition}

\begin{proof}
Let us denote
$$
\lambda=\frac{1}{2}(\tau^2+v_i^2+d_{\alpha}^2+m_{A,I}^2-1),\;\;\;
\omega_i=\tau v_i-m_{\alpha i}d_{\alpha}
$$
$$
\varphi_{A,I}^{\alpha}=\sum_{i\in I}(-1)^{O_A(\alpha)+O_I(i)}m_{A,I\setminus\{i\}}m_{\alpha i} - \mathbbm{1}_{\{\alpha\notin A\}}\tau m_{A\bigcup\{\alpha\},I}
$$
$$
\psi_{A,I}^{i}=\sum_{\alpha\in A}(-1)^{O_A(\alpha)+O_I(i)}m_{A\setminus\{\alpha\},I}m_{\alpha i} - \mathbbm{1}_{\{i\notin I\}}\tau m_{A,I\bigcup\{i\}}
$$
It is enough to show that $\lambda$, $\omega_i$, $\varphi_{A,I}^{\alpha}$,$\psi_{A,I}^{i}$ always vanish. Since $\varphi_{A,I}^{\alpha}$,$\psi_{A,I}^{i}$ satisfy \eqref{eq:cons-3} and \eqref{eq:cons-4}, which are linear symmetric system of PDEs when we see $(\tau,v_i,d_{\alpha},m_{A,I})$ as fixed functions. It is easy to know that 0 is the unique solution when initial data is 0. So we get $\varphi_{A,I}^{\alpha}=\psi_{A,I}^{i}=0$ for all possible choice of $A,I,\alpha,i$. Therefore, we know that $m_{A,I}=\tau[F]_{A,I}$, where $F_{\alpha i}=\tau^{-1}m_{\alpha i}$. So, by \eqref{eq:cons-1} and \eqref{eq:cons-2}, we know that $\lambda$, $\omega_i$ solves the following linear system of PDEs
\begin{equation}
D_t \lambda = \tau Z_{ij}\partial_j\omega_i + f_i \omega_i
\end{equation}
\begin{equation}
D_t \omega_i = \tau\partial_i \lambda + c_{ij} \omega_j
\end{equation}
where $D_t=\partial_t + v\cdot\nabla$, $Z_{ij}=\xi(F)(I_n+F^TF)^{-1}_{ij}$ is a positive definite matrix, $f_i=\partial_j Z_{ij}$, $c_{ij}=\delta_{ij}\nabla\cdot v-\partial_i v_j$. This system is of hyperbolic type and looks very like the acoustic waves. Now, since $Z_{ij}$ is positive definite, we can find a positive definite matrix $Q$ such that $Z=Q^2$. Now we do the change of variable $\widetilde{\omega}_{i}=Q_{ij}\omega_j$, then $\lambda$, $\widetilde{\omega}_{i}$ should solve the following linear symmetric system of PDEs
\begin{equation}
D_t \lambda = \tau Q_{ij}\partial_j\widetilde{\omega}_i + \widetilde{f}_i \widetilde{\omega}_i
\end{equation}
\begin{equation}
D_t \widetilde{\omega}_i = \tau Q_{ij}\partial_j \lambda + \widetilde{c}_{ij} \widetilde{\omega}_j
\end{equation}
where
$$\widetilde{f}_i=\tau Z_{jk}\partial_kQ^{-1}_{ij} + Q^{-1}_{ij}f_j,\;\;\;\widetilde{c}_{ij}=(D_t Q_{ik} + Q_{il}c_{lk})Q^{-1}_{kj}$$
By the standard method of analysis of PDEs, it is easy to know that this linear symmetric system has a unique solution. Since $\lambda=\widetilde{\omega}_i=0$ at $t=0$, so we have $\lambda\equiv\widetilde{\omega}_i\equiv0$. So we have $\omega_i\equiv0$, which completes the proof.

\end{proof}


\section{Toward mean curvature motions in the Euclidean space}

We conclude this paper by explaining how mean curvature motions in the
Euclidean space are related to our study of extremal surfaces in the Minkowski space.
This can be done very simply by the elementary
quadratic change of time $\theta=t^2/2$ in the extremal surface equations \eqref{eq:NG}. Let us work in the case where $X^{0}(t,x)=t$. We do the change of coordinate $\theta=t^2/2$, and in the new coordinate system, the extremal surface is denoted by $X^{M}(\theta,x)$. The chain rule tells us
$$\partial_{t}X^0\equiv1,\quad\partial_{t}X^{M}=\theta'\partial_{\theta}X^{M},\quad M=1,\ldots,m+n$$
Now for fixed $\theta$, the slice of $X(\theta,x)=(X^1(\theta,x),\ldots,X^{m+n}(\theta,x))$ is a $n$ dimensional manifold $\Sigma$ in $\mathbb{R}^{m+n}$. Let us denote the induced metric on $\Sigma$ by $g_{ij}=\langle\partial_{i}X,\partial_{j}X\rangle$, $i,j=1,\ldots,n$. Denote $g=\det{g_{ij}}$, $g^{ij}$ the inverse of $g_{ij}$.
Then we can get that
$$G_{00}=-1+{\theta'}^2|\partial_{\theta}X|^2,\; G_{0i}=G_{i0}=\theta'\langle\partial_{\theta}X,\partial_{i}X\rangle=\theta'h_i,\; G_{ij}=g_{ij}$$
$$G=\det\left(
                \begin{array}{cc}
                  -1+{\theta'}^2|\partial_{\theta}X|^2 & \theta'h_j \\
                  \theta'h_i & g_{ij} \\
                \end{array}
              \right)=-\left[1+ 2\theta\left(h_ih_jg^{ij}-|\partial_{\theta}X|^2\right)\right]g$$
$$\sqrt{-G}=\sqrt{g}\left[1+\theta\left(h_ih_jg^{ij}-|\partial_{\theta}X|^2\right) + \mathcal{O}(\theta^2)\right]  $$
$$G^{00}=-1 + 2\theta\left(h_ih_jg^{ij}-|\partial_{\theta}X|^2\right)+ \mathcal{O}(\theta^2),\;G^{0i}=G^{i0}=\theta'g^{ij}h_j+ \mathcal{O}(\theta),\;G^{ij}=g^{ij} + \mathcal{O}(\theta)$$
Therefore, \eqref{eq:NG} can be rewritten as
\begin{equation*}
\begin{array}{r@{}l}
{0 \;} & \displaystyle{ =\partial_{t}\left(\sqrt{-G}G^{00}\right) + \partial_{i}\left(\sqrt{-G}G^{i0}\right) }\\
& \displaystyle{= \theta'\left[-\partial_{\theta}\left(\sqrt{g}\right) + \sqrt{g}\left(h_ih_jg^{ij}-|\partial_{\theta}X|^2\right)+ \partial_{i}\left(\sqrt{g}g^{ij}h_j\right) \right] + \mathcal{O}(\theta) }
\end{array}
\end{equation*}
\begin{equation*}
\begin{array}{r@{}l}
{0\;} & \displaystyle{= \partial_{t}\left(\sqrt{-G}G^{00}\partial_{t}X^{M}\right) + \partial_{i}\left(\sqrt{-G}G^{i0}\partial_{t}X^{M}\right)+\partial_{i}\left(\sqrt{-G}G^{ij}\partial_{j}X^{M}\right) }\\
& \displaystyle{=  -\partial_{t}\left(\theta'\sqrt{g}\partial_{\theta}X^{M}\right) + \partial_{i}\left({\theta'}^2\sqrt{g}g^{ij}h_j\partial_{\theta}X^{M}\right)+\partial_{i}\left(\sqrt{g}g^{ij}\partial_{j}X^{M}\right) + \mathcal{O}(\theta)}\\
& \displaystyle{= -\sqrt{g}\partial_{\theta}X^{M} +  \partial_{i}\left(\sqrt{g}g^{ij}\partial_{j}X^{M}\right) + \mathcal{O}(\theta) }
\end{array}
\end{equation*}
In the regime $\theta\ll1$, we have the following equations
\begin{equation}\label{eq:mcf-g}
\partial_{\theta}\left(\sqrt{g}\right) +\sqrt{g}|\partial_{\theta}X|^2 = \partial_{i}\left(\sqrt{g}g^{ij}h_j\right) +\sqrt{g}h_ih_jg^{ij}
\end{equation}
\begin{equation}\label{eq:mcf}
\partial_{\theta}X^{M}=\frac{1}{\sqrt{g}}\partial_{i}\left(\sqrt{g}g^{ij}\partial_{j}X^{M}\right),\qquad M=1,
\ldots,m+n
\end{equation}
\eqref{eq:mcf} is exactly the equation for the $n$ dimensional mean curvature flow in $\mathbb{R}^{m+n}$, and \eqref{eq:mcf-g} is just a consequence of \eqref{eq:mcf}.
\begin{remark}
It can be easily shown that \eqref{eq:mcf} is equivalent to the following equation
\begin{equation}
\partial_{\theta}X^{M}=g^{ij}\partial_{ij}X^{M}-g^{ij}g^{kl}\partial_{k}X^{M}\partial_{l}X^{N}\partial_{ij}X_{N}
\end{equation}
Therefore,
\begin{equation*}
\begin{array}{r@{}l}
 \displaystyle{ h_i  = \partial_{\theta} X^M \partial_{i} X_{M} } & \displaystyle{ = \left(g^{jk}\partial_{jk}X^{M}-g^{jk}g^{lm}\partial_{l}X^{M}\partial_{m}X^{N}
\partial_{jk}X_{N}\right)\partial_{i} X_{M}}\\
& \displaystyle{= g^{jk}\partial_{jk}X^{M}\partial_{i} X_{M} - g^{jk}g^{lm}g_{il}\partial_{m}X^N \partial_{jk} X_N}\\
& =0
\end{array}
\end{equation*}
As a consequence, we have
\begin{equation*}
\begin{array}{r@{}l}
 \displaystyle{ \partial_{\theta}\left(\sqrt{g}\right) } & \displaystyle{ = \frac{1}{\sqrt{g}}gg^{ij}\partial_{i\theta}X^M\partial_{j}X_{M}}\\
& \displaystyle{=\partial_{i}(\sqrt{g}g^{ij}\partial_{\theta}X^M\partial_{j}X_{M}) - \partial_{\theta}X^M \partial_{i}\left(\sqrt{g}g^{ij}\partial_{j}X_{M}\right)}\\
& \displaystyle{= -\sqrt{g}|\partial_{\theta}X|^2 }
\end{array}
\end{equation*}
which is exactly \eqref{eq:mcf-g} since $h_i=0$, $i=1,\ldots,n$.
\end{remark}
So, we may expect to perform for the mean-curvature flow the same type of analysis we did
for the extremal surfaces, which we intend to do in a future work.


\section{Appendix A}\label{apdx-a}

Let us denote
$$F^{\alpha}_{\ i}=\partial_{i}X^{p+\alpha},\;\;V^{\alpha}=\partial_{t}X^{p+\alpha},\;\;\xi_{ij}=\delta_{ij}+F^{\alpha}_{\ i}F_{\alpha j},\;\;\zeta_{\alpha\beta}=\delta_{\alpha\beta} + F_{\alpha}^{\ i}F_{\beta i},$$
and let $\xi^{ij},\zeta^{\alpha\beta}$ be respectively the inverse of $\xi_{ij}, \zeta_{\alpha\beta}$, $\xi=\det\xi_{ij}=\det{\zeta_{\alpha\beta}}$, $i,j=1,\ldots,n$, $\alpha,\beta=1,\ldots,m$.
Since $\xi_{ij}F_{\alpha}^{\ j}=F_{\alpha i}+F^{\beta}_{\ i}F_{\beta j}F_{\alpha}^{\ j}=F^{\beta}_{\ i}\zeta_{\beta \alpha}$, we have $\xi^{ij}F_{\alpha j}=\zeta^{\alpha\beta}F_{\beta i}$. By using the above notations, the induced metric $G_{\mu\nu}$ has the following expression,
$$(G_{\mu\nu})=\left(
                \begin{array}{cc}
                  -1+|V|^2 & F_{\ j}^{\alpha}V_{\alpha} \\
                  F_{\ i}^{\alpha}V_{\alpha}  & \xi_{ij} \\
                \end{array}
              \right),\ G=-\xi\left(1-|V|^2+\xi^{ij}F^{\alpha}_{\ i}F^{\beta}_{\ j}V_{\alpha}V_{\beta}\right)=-\xi\left(1-\zeta^{\alpha\beta}
V_{\alpha}V_{\beta}\right)$$
$$(G^{\mu\nu})=G^{-1}\xi\left(
                \begin{array}{cc}
                  1 & -\zeta^{\alpha\beta}V_{\alpha}F_{\beta j} \\
                  -\zeta^{\alpha\beta}V_{\alpha}F_{\beta i}  & (-1+|V|^2)\xi^{ij} +(\xi^{ik}\xi^{jl}-\xi^{ij}\xi^{kl})F_{\ k}^{\alpha}F_{\ l}^{\beta}V_{\alpha}V_{\beta} \\
                \end{array}
              \right)$$
Now let's start looking at the equation \eqref{eq:NG}. The equation for $X^{i},i=1,\ldots,n$, reads
$$\partial_{t}\left(\frac{\sqrt{\xi}\zeta^{\alpha\beta}
V_{\alpha}F_{\beta i}}{\sqrt{1-\zeta^{\alpha\beta}
V_{\alpha}V_{\beta}}}\right) - \partial_{j}\left\{\frac{\sqrt{\xi}\left[(-1+|V|^2)\xi^{ij} +(\xi^{ik}\xi^{jl}-\xi^{ij}\xi^{kl})F_{\ k}^{\alpha}F_{\ l}^{\beta}V_{\alpha}V_{\beta}\right]}{\sqrt{1-\zeta^{\alpha\beta}
V_{\alpha}V_{\beta}}}\right\}=0$$
We denote
$$D^{\alpha}=\frac{-\sqrt{\xi}\zeta^{\alpha\beta}V_{\beta}}{\sqrt{1-\zeta^{\beta\gamma}
V_{\beta}V_{\gamma}}},\quad P_{i}=f^{\alpha}_{\ i}D_{\alpha},\quad h=\frac{\sqrt{\xi}}{\sqrt{1-\zeta^{\alpha\beta}
V_{\alpha}V_{\beta}}}$$
Then we have
$$V^{\alpha}=\frac{-D^{\alpha}-F^{\alpha}_{\ i}P^{i}}{\sqrt{\xi +|D|^2 + |P|^2 }},\quad h=\sqrt{\xi +|D|^2 + |P|^2 }$$
Therefore, the equation can be rewritten as
\begin{equation}\label{eq:poly-1}
\partial_{t}P_{i} + \partial_{j}\left(\frac{P_iP_j-\xi\xi^{ij}}{h}\right)=0
\end{equation}
The equation for $X^0=t$ reads,
$$-\partial_{t}\left(\frac{\sqrt{\xi}}{\sqrt{1-\zeta^{\alpha\beta}
V_{\alpha}V_{\beta}}}\right) + \partial_{j}\left(\frac{\sqrt{\xi}\zeta^{\alpha\beta}
V_{\alpha}F_{\beta j}}{\sqrt{1-\zeta^{\alpha\beta}
V_{\alpha}V_{\beta}}}\right)=0$$
which can be rewritten by using our new notations as
\begin{equation}\label{eq:poly-2}
\partial_t h + \partial_j P_j =0
\end{equation}
The equation for $X^{p+\alpha}$, $\alpha=1,\ldots,m$, reads,
\begin{multline*}
- \partial_{t} \left(\frac{\sqrt{\xi}(V_{\alpha}-\zeta^{\beta\gamma}F_{\alpha}^{\ j}F_{\beta j}V_{\gamma})}{\sqrt{1-\zeta^{\alpha\beta}
V_{\alpha}V_{\beta}}}\right) + \partial_j\left(\frac{\sqrt{\xi}(\zeta^{\beta\gamma}F_{\beta}^{\ j}V_{\gamma}V_{\alpha} - F_{\alpha i}\xi^{ik}F_{\ k}^{\beta}V_{\beta}\xi^{jl}F_{\ l}^{\gamma}V_{\gamma})}{\sqrt{1-\zeta^{\alpha\beta}
V_{\alpha}V_{\beta}}}\right) \\ +\partial_j\left(\sqrt{\xi}\sqrt{1-\zeta^{\alpha\beta}
V_{\alpha}V_{\beta}}\xi^{ij}F_{\alpha i}\right)=0
\end{multline*}
which can be rewritten as
\begin{equation}\label{eq:poly-3}
\partial_{t} D_{\alpha} +\partial_j \left(\frac{D_{\alpha}P_{j}+\xi\xi^{ij}F_{\alpha i}}{h}\right)=0
\end{equation}
At last, since $\partial_{t}F_{\alpha i}=\partial_{i}V_{\alpha}$, $ \partial_{i}F_{\alpha j}=\partial_{j}F_{\alpha i}$, we have
\begin{equation}\label{eq:poly-4}
\partial_t F_{\alpha i} + \partial_i\left(\frac{D_{\alpha}+F_{\alpha j}P^{j}}{h}\right)=0,\ \ \  \partial_{i}F_{\alpha j}=\partial_{j}F_{\alpha i}
\end{equation}
\eqref{eq:poly-1}-\eqref{eq:poly-4} are just the equations that we propose.


\section{Appendix B}\label{apdx-b}
First, let's prove the equation \eqref{eq:so4}. Quite directly, we have
$$
\partial_t h=\frac{1}{2h}\left(2D_{\alpha}\partial_t D_{\alpha}+2P_i\partial_t P_i + \frac{\partial\xi(F)}{\partial F_{\alpha i}}\partial_t F_{\alpha i}\right)
$$
$$
=\frac{D_{\alpha}\partial_t D_{\alpha}+P_i\partial_t (F_{\alpha i}D_{\alpha}) + \xi'(F)_{\alpha i}\partial_t F_{\alpha i}}{h}
$$
$$
=\left(\frac{D_{\alpha} + F_{\alpha i}P_i}{h}\right)\partial_t D_{\alpha} + \left(\frac{D_{\alpha}P_i + \xi'(F)_{\alpha i}}{h}\right)\partial_t F_{\alpha i}
$$
$$
=-\left(\frac{D_{\alpha} + F_{\alpha j}P_j}{h}\right)\partial_i\left(\frac{D_{\alpha}P_i+\xi'(F)_{\alpha i}}{h}\right)-\left(\frac{D_{\alpha}P_i + \xi'(F)_{\alpha i}}{h}\right)\partial_i\left(\frac{D_{\alpha}+F_{\alpha j}P_j}{h}\right)
$$
$$
=-\partial_i\left(\frac{(D_{\alpha} + F_{\alpha j}P_j)(D_{\alpha}P_i + \xi'(F)_{\alpha i})}{h^2}\right)
$$
Now, since
$$
\xi'(F)_{\alpha i}(D_{\alpha} + F_{\alpha j}P_j)=\xi(F)(I+F^TF)^{-1}_{ik}F_{\alpha k}(D_{\alpha} + F_{\alpha j}P_j)
$$
$$
=\xi(F)(I+F^TF)^{-1}_{ik}(P_k + F_{\alpha k}F_{\alpha j}P_j)
$$
$$
=\xi(F)(I+F^TF)^{-1}_{ik}(I+F^TF)_{kj}P_j=\xi(F)\delta_{ij}P_j=\xi(F)P_i
$$
So we have
$$
\partial_t h=-\partial_i\left(\frac{(D^2+P^2+\xi(F))P_i}{h^2}\right)=-\partial_i P_i
$$
Now, let's look at the equation for $P_i=F_{\alpha i}D_{\alpha}$. We have
$$
\partial_t P_i= \partial_t (F_{\alpha i}D_{\alpha})=D_{\alpha}\partial_t F_{\alpha i} + F_{\alpha i}\partial_t D_{\alpha}
$$
The first term
$$
D_{\alpha}\partial_t F_{\alpha i}=-D_{\alpha}\partial_i\left(\frac{D_{\alpha}+F_{\alpha j}P_j}{h}\right)
$$
$$
=-\partial_i\left(\frac{D_{\alpha}^2+ D_{\alpha}F_{\alpha j}P_j}{h}\right) + \frac{D_{\alpha}\partial_iD_{\alpha}+P_j F_{\alpha j}\partial_iD_{\alpha}}{h}
$$
$$
=-\partial_i\left(\frac{D^2+P^2}{h}\right) + \frac{D_{\alpha}\partial_i D_{\alpha}+P_j\partial_i (F_{\alpha j}D_{\alpha}) + \xi'(F)_{\alpha j}\partial_i F_{\alpha j}}{h} -\frac{P_jD_{\alpha}\partial_iF_{\alpha j}+\xi'(F)_{\alpha j}\partial_i F_{\alpha j}}{h}
$$
$$
=-\partial_i\left(\frac{h^2-\xi(F)}{h}\right)+\partial_i h -\left(\frac{D_{\alpha}P_j+\xi'(F)_{\alpha j}}{h}\right)\partial_iF_{\alpha j}
$$
$$
=\partial_i\left(\frac{\xi(F)}{h}\right)-\left(\frac{D_{\alpha}P_j+\xi'(F)_{\alpha j}}{h}\right)\partial_iF_{\alpha j}
$$
The second term
$$
F_{\alpha i}\partial_t D_{\alpha}=-F_{\alpha i}\partial_j\left(\frac{D_{\alpha}P_j+\xi'(F)_{\alpha j}}{h}\right)
$$
$$
=-\partial_j\left(\frac{(F_{\alpha i}D_{\alpha})P_j+F_{\alpha i}\xi'(F)_{\alpha j}}{h}\right)+ \left(\frac{D_{\alpha}P_j+\xi'(F)_{\alpha j}}{h}\right)\partial_jF_{\alpha i}
$$
$$
=-\partial_j\left(\frac{P_iP_j+F_{\alpha i}\xi'(F)_{\alpha j}}{h}\right)+ \left(\frac{D_{\alpha}P_j+\xi'(F)_{\alpha j}}{h}\right)\partial_iF_{\alpha j}
$$
So we have
$$
\partial_t P_i=-\partial_j\left(\frac{P_iP_j+F_{\alpha i}\xi'(F)_{\alpha j}-\xi(F)\delta_{ij}}{h}\right)
$$
Now, since
$$
(I+F^TF)_{ik}^{-1}(\delta_{jk}+F_{\alpha j}F_{\alpha k})=\delta_{ij}
$$
then we have
$$
F_{\alpha i}\xi'(F)_{\alpha j}-\xi(F)\delta_{ij}=\xi(F)((I+F^TF)_{ik}^{-1}F_{\alpha j}F_{\alpha k}-\delta_{ij})=-\xi(F)(I+F^TF)_{ij}^{-1}
$$
So we get
$$\partial_tP_i + \partial_j\left(\frac{P_i P_j}{h}-\frac{\xi(F)(I+F^T F)^{-1}_{ij}}{h}\right)=0$$


\section{Appendix: C}\label{apdx-c}

We have that
$$
\partial_t S=\frac{D_{\alpha}\partial_t D_{\alpha} + P_i\partial_t P_i + M_{A,I}\partial_t M_{A,I}}{h}-\frac{1+D_{\alpha}^2 + P_i^2 + M_{A,I}^2}{2h^2}\partial_t h
$$
Let's look at the first term,
$$
\frac{D_{\alpha}\partial_t D_{\alpha}}{h}=-\frac{D_{\alpha}}{h}\partial_j\left(\frac{D_{\alpha}P_j}{h}\right)  -\frac{D_{\alpha}}{h} \sum_{A,I,i\atop\alpha\in A,i\in I}(-1)^{O_A(\alpha)+O_I(i)}\partial_i\left(\frac{M_{A,I}M_{A\setminus \{\alpha\},I\setminus\{i\} }}{h}\right)
$$
(since \eqref{eq:sn5}, we have $\sum_{i\in I}(-1)^{O_A(\alpha)+O_I(i)}\partial_iM_{A\setminus \{\alpha\},I\setminus\{i\}}=0$, )
$$
=-\frac{D_{\alpha}^2}{h^2}\partial_j P_j -P_j\partial_j\left(\frac{D_{\alpha}^2}{2h^2}\right)-\frac{D_{\alpha}M_{A\setminus \{\alpha\},I\setminus\{i\}}}{h} \sum_{A,I,i\atop\alpha\in A,i\in I}(-1)^{O_A(\alpha)+O_I(i)}\partial_i\left(\frac{M_{A,I} }{h}\right)
$$
For the second term,
$$
\frac{P_i\partial_t P_i}{h}=L_1+L_2
$$
where
$$L_1=-\frac{P_i}{h}\sum_{A,I,j\atop j\in I,i\notin I\setminus\{j\}}(-1)^{O_{I}(j)+O_{I\setminus\{j\}}(i)}\partial_j\left(\frac{M_{A,(I\setminus\{j\})\bigcup\{i\}}M_{A,I}}{h}\right)
$$
$$
L_2=-\frac{P_i}{h}\partial_j\left(\frac{P_i P_j}{h}\right) +\frac{P_j}{h}\partial_j\left(\frac{1+\sum_{A,I}M_{A,I}^2}{h}\right)
$$
Now let's first prove the following equality,
$$
\partial_i M_{A,I}=\sum_{j\in I}\mathbbm{1}_{\{i\notin I\setminus\{j\}\}}(-1)^{O_I(j)+O_{I\setminus\{j\}}(i)}\partial_{j}M_{A,(I\setminus\{j\})\bigcup\{i\}}
$$
In fact, since $\mathbbm{1}_{\{j\in I,i\notin I\setminus\{j\}\}}=\mathbbm{1}_{\{j\in I,i\notin I\}}+\mathbbm{1}_{\{j\in I,i=j\}}$, we have
$${\rm Right\ }= \mathbbm{1}_{\{i\notin I\}}\sum_{j\in I}(-1)^{O_I(j)+O_{I\setminus\{j\}}(i)}\partial_{j}M_{A,(I\bigcup\{i\})\setminus\{j\}}+ \mathbbm{1}_{\{i\in I\}}\partial_i M_{A,I}$$
For $i\neq j$, we can check that
$$O_I(j)+O_{I\setminus\{j\}}(i) \equiv  O_I(i)+O_{I\bigcup\{i\}}(j) + 1\;\;({\rm mod\ }2)$$
So the right hand side
$${\rm Right\ }= -\mathbbm{1}_{\{i\notin I\}}\sum_{j\in I}(-1)^{O_I(i)+O_{I\bigcup\{i\}}(j)}\partial_{j}M_{A,(I\bigcup\{i\})\setminus\{j\}}+ \mathbbm{1}_{\{i\in I\}}\partial_i M_{A,I}$$
$$=-\mathbbm{1}_{\{i\notin I\}}\sum_{j\in I\bigcup\{i\}}(-1)^{O_I(i)+O_{I\bigcup\{i\}}(j)}\partial_{j}M_{A,(I\bigcup\{i\})\setminus\{j\}}+ \big(\mathbbm{1}_{\{i\notin I\}}+\mathbbm{1}_{\{i\in I\}}\big)\partial_i M_{A,I}$$
Because of \eqref{eq:sn5}, we finally get
$$\sum_{j\in I}\mathbbm{1}_{\{i\notin I\setminus\{j\}\}}(-1)^{O_I(j)+O_{I\setminus\{j\}}(i)}\partial_{j}M_{A,(I\setminus\{j\})\bigcup\{i\}}=\partial_i M_{A,I}$$
So we have
$$L_1=-\sum_{A,I}\frac{P_i M_{A,I}}{h^2}\partial_i M_{A,I}-\frac{P_i M_{A,(I\setminus\{j\})\bigcup\{i\}}}{h}\sum_{A,I,j\atop j\in I,i\notin I\setminus\{j\}}(-1)^{O_{I}(j)+O_{I\setminus\{j\}}(i)}\partial_j\left(\frac{M_{A,I}}{h}\right)$$
For $L_2$, we have
$$L_2= -\frac{P_i^2}{h^2}\partial_j P_j - P_j\partial_j\left(\frac{P_i^2}{2h^2}\right) + \partial_j\left(\frac{P_j(1+M_{A,I}^2)}{h^2}\right)- \frac{1+M_{A,I}^2}{h}\partial_j\left(\frac{P_j}{h}\right)$$
Since
$$- \frac{1+M_{A,I}^2}{h}\partial_j\left(\frac{P_j}{h}\right)=- \frac{1+M_{A,I}^2}{h^2}\partial_jP_j-P_j\partial_j\left(\frac{1+M_{A,I}^2}{2h^2}\right)+ \frac{P_j M_{A,I}}{h^2}\partial_j M_{A,I}$$
so we have
$$
\frac{P_i\partial_tP_i}{h}=-\frac{P_i M_{A,(I\setminus\{j\})\bigcup\{i\}}}{h}\sum_{A,I,j\atop j\in I,i\notin I\setminus\{j\}}(-1)^{O_{I}(j)+O_{I\setminus\{j\}}(i)}\partial_j\left(\frac{M_{A,I}}{h}\right) $$
$$- \frac{1+P_i^2+ M_{A,I}^2}{h^2}\partial_jP_j-P_j\partial_j\left(\frac{1+P_i^2+M_{A,I}^2}{2h^2}\right)+ \partial_j\left(\frac{P_j(1+M_{A,I}^2)}{h^2}\right) $$
Therefore, we have
$$\frac{D_{\alpha}\partial_t D_{\alpha} + P_i\partial_t P_i + M_{A,I}\partial_t M_{A,I}}{h}=-\frac{2S}{h}\partial_jP_j-P_j\partial_j\left(\frac{S}{h}\right)+ L_3$$
where
$$L_3=\partial_j\left(\frac{P_j(1+M_{A,I}^2)}{h^2}\right)-\sum_{A,I,j\atop j\in I,i\notin I\setminus\{j\}}(-1)^{O_{I}(j)+O_{I\setminus\{j\}}(i)}\partial_j\left(\frac{P_i M_{A,(I\setminus\{j\})\bigcup\{i\}}M_{A,I}}{h^2}\right)$$
$$-\sum_{A,I,i\atop\alpha\in A,i\in I}(-1)^{O_A(\alpha)+O_I(i)}\partial_i\left(\frac{D_{\alpha}M_{A\setminus \{\alpha\},I\setminus\{i\}}M_{A,I} }{h^2}\right)$$
So we have
$$\partial_t S= -\frac{2S}{h}\partial_jP_j-P_j\partial_j\left(\frac{S}{h}\right)+ L_3 -\frac{S}{h}\partial_t h=-\partial_j\left(\frac{SP_j}{h}\right)+L_3$$
which completes the proof.


\section{Appendix D}\label{apdx-d}

\subsection*{Equation for $\omega_i$}

First, let's compute $\partial_t \omega_i$, by definition,
$$
\partial_t \omega_i =\partial_t \tau v_i +\tau \partial_t v_i - \partial_t m_{\alpha i}d_{\alpha}-m_{\alpha i}\partial_td_{\alpha}
$$
The first two terms are,
$$
\partial_t \tau v_i +\tau \partial_t v_i=v_i(\tau\partial_jv_j-v_j\partial_j\tau) +\tau\Big(
\sum_{A,I}m_{A,I}\partial_{i}m_{A,I} - v_j\partial_j v_i + \tau\partial_i\tau\Big)+\Sigma_1
$$
$$
=(\tau v_i)\partial_jv_j - v_j\partial_j(\tau v_i) +\frac{\tau}{2}\partial_i(\tau^2+m^2_{A,I})+\Sigma_1
$$
where
$$
\Sigma_1=-\tau\sum_{A,I,j}\mathbbm{1}_{\{j\in I,i\notin I\setminus\{j\}\}}(-1)^{O_{I}(j)+O_{I\setminus\{j\}}(i)}m_{A,(I\setminus\{j\})\bigcup\{i\}}\partial_jm_{A,I}
$$
Here we use the equation for $m_{\alpha i}$:
$$
\partial_t m_{\alpha i} + v_j\partial_jm_{\alpha i} +  m_{\alpha j}\partial_{i}v_j  - m_{\alpha i}\partial_j v_j +   \tau\partial_{i}d_{\alpha} =0
$$
The last two terms are,
$$
- \partial_t m_{\alpha i}d_{\alpha}-m_{\alpha i}\partial_td_{\alpha}=d_{\alpha}(v_j\partial_j m_{\alpha i} + m_{\alpha j}\partial_i v_j-m_{\alpha i}\partial_j v_j +\tau\partial_i d_{\alpha})+m_{\alpha i}v_j\partial_j d_{\alpha} + \Sigma_2
$$
$$
=-(m_{\alpha i}d_{\alpha})\partial_j v_j + v_j\partial_j(m_{\alpha i}d_{\alpha}) + \frac{\tau}{2}\partial_i(d^2_{\alpha}+v^2_j)-(\tau v_j-m_{\alpha j}d_{\alpha})\partial_i v_j+\Sigma_2
$$
where
$$
\Sigma_2=\sum_{A,I,\alpha,j}\mathbbm{1}_{\{\alpha\in A,j\in I\}}(-1)^{O_A(\alpha)+O_I(j)}m_{\alpha i}m_{A\setminus \{\alpha\},I\setminus\{j\} }\partial_j m_{A,I}
$$
Now, we have
$$
\partial_t \omega_i=(\tau v_i-m_{\alpha i}d_{\alpha})\partial_j v_j -v_j\partial_j(\tau v_i-m_{\alpha i}d_{\alpha}) + \frac{\tau}{2}\partial_i(\tau^2+v_j^2+d_{\alpha}^2+m_{A,I}^2)
$$
$$
-(\tau v_j-m_{\alpha j}d_{\alpha})\partial_i v_j+\Sigma_1+\Sigma_2
$$
$$
=\omega_i\partial_j v_j-\omega_j\partial_i v_j-v_j\partial_j\omega_i +\tau\partial_i\lambda +\Sigma_1+\Sigma_2
$$
It is easy to check that
$$
\Sigma_1+\Sigma_2=\sum_{A,I,j}\mathbbm{1}_{\{j\in I\}}(-1)^{O_{I}(j)+O_{I\setminus\{j\}}(i)}\partial_j m_{A,I}\psi^{i}_{A,I\setminus\{j\}}
$$
So $\omega_i$ should satisfy the following equation
$$
\partial_t \omega_i=\omega_i\partial_j v_j-\omega_j\partial_i v_j-v_j\partial_j\omega_i +\tau\partial_i\lambda +\sum_{A,I,j}\mathbbm{1}_{\{j\in I\}}(-1)^{O_{I}(j)+O_{I\setminus\{j\}}(i)}\partial_j m_{A,I}\psi^{i}_{A,I\setminus\{j\}}
$$


\subsection*{Equation for $\lambda$}
Now let's compute $\partial_t\lambda$, we have,
$$
\partial_t\lambda=\tau\partial_t\tau + v_i\partial_t v_i + d_{\alpha}\partial_t d_{\alpha}  + m_{A,I}\partial_t m_{A,I}
$$
$$
=\tau(\tau\partial_jv_j-v_j\partial_j\tau) + v_i( m_{A,I}\partial_{i}m_{A,I} - v_j\partial_j v_i + \tau\partial_i\tau) + \Sigma_3
$$
$$
- d_{\alpha}v_j\partial_j d_{\alpha} +\Sigma_4 + m_{A,I}(m_{A,I}\partial_j v_j-v_j\partial_j m_{A,I}) +\Sigma_5 +\Sigma_6
$$
$$
=-\frac{v_j}{2}\partial_j(\tau^2+v_i^2+d_{\alpha}^2+m_{A,I}^2) +\tau\partial_i(\tau v_i) + m_{A,I}\partial_i(m_{A,I}v_i)
$$
$$
+ \Sigma_3+ \Sigma_4+ \Sigma_5+ \Sigma_6
$$
where
$$
\Sigma_3=-\sum_{A,I,i,j}\mathbbm{1}_{\{j\in I,i\notin I\setminus\{j\}\}}(-1)^{O_{I}(j)+O_{I\setminus\{j\}}(i)} m_{A,(I\setminus\{j\})\bigcup\{i\}}v_i\partial_jm_{A,I}
$$
$$
\Sigma_4=-\sum_{A,I,\alpha,i}\mathbbm{1}_{\{\alpha\in A,i\in I\}}(-1)^{O_A(\alpha)+O_I(i)}m_{A\setminus \{\alpha\},I\setminus\{i\} }d_{\alpha}\partial_im_{A,I}
$$
$$
\Sigma_5=-\sum_{A,I,\alpha,i}\mathbbm{1}_{\{\alpha\in A,i\in I\}}(-1)^{O_A(\alpha)+O_I(i)} m_{A\setminus\{\alpha\},I\setminus\{i\}} m_{A,I}\partial_{i}d_{\alpha}
$$
$$
\Sigma_6=-\sum_{A,I,i,j}\mathbbm{1}_{\{j\in I,i\notin I\setminus\{j\}\}}(-1)^{O_{I\setminus\{j\}}(i)+O_{I}(j)} m_{A,(I\setminus\{j\})\bigcup\{i\}}m_{A,I}\partial_{j}v_i
$$
It is easy to see that
$$
\Sigma_3+\Sigma_6=-\sum_{A,I,i,j}\mathbbm{1}_{\{j\in I,i\notin I\setminus\{j\}\}}(-1)^{O_{I}(j)+O_{I\setminus\{j\}}(i)} m_{A,(I\setminus\{j\})\bigcup\{i\}}\partial_j (m_{A,I}v_i)
$$
(since $\mathbbm{1}_{\{j\in I,i\notin I\setminus\{j\}\}}=\mathbbm{1}_{\{j\in I,i\notin I\}}+\mathbbm{1}_{\{j\in I,i=j\}}$)
$$
=-\sum_{A,I,i}\mathbbm{1}_{\{i\in I\}}m_{A,I}\partial_i(m_{A,I}v_i)-\sum_{A,I,i,j}\mathbbm{1}_{\{j\in I,i\notin I\}}(-1)^{O_{I}(j)+O_{I\setminus\{j\}}(i)} m_{A,(I\setminus\{j\})\bigcup\{i\}}\partial_j (m_{A,I}v_i)
$$
(since $\mathbbm{1}_{\{i\in I\}}=1-\mathbbm{1}_{\{i\notin I\}}$)
$$
=-m_{A,I}\partial_i(m_{A,I}v_i)+\sum_{A,I,i,j}\mathbbm{1}_{\{i\notin I,j=i\}}(-1)^{O_I(i)+O_{I\bigcup\{i\}}(j)}m_{A,(I\bigcup\{i\})\setminus\{j\}}\partial_j (m_{A,I}v_i)
$$
$$
-\sum_{A,I,i,j}\mathbbm{1}_{\{i\notin I,j\in I\}}(-1)^{O_{I}(j)+O_{I\setminus\{j\}}(i)} m_{A,(I\bigcup\{i\})\setminus\{j\}}\partial_j (m_{A,I}v_i)
$$
Now we can easily check that for any $i\notin I$ and $j\in I$,
$$O_I(i)+O_{I\bigcup\{i\}}(j)\equiv O_I(j)+O_{I\setminus\{j\}}(i)+1\;\;({\rm mod\ }2)$$
[We can prove this equality by discussing in the cases $i<j$ and $i>j$.]\\
Because $\mathbbm{1}_{\{i\notin I,j=i\}}+\mathbbm{1}_{\{i\notin I,j\in I\}}=\mathbbm{1}_{\{i\notin I,j\in I\bigcup\{i\}\}}$, we have
$$
\Sigma_3+\Sigma_6 + m_{A,I}\partial_i(m_{A,I}v_i)=\sum_{A,I,i,j}\mathbbm{1}_{\{i\notin I,j\in I\bigcup\{i\}\}}(-1)^{O_I(i)+O_{I\bigcup\{i\}}(j)}m_{A,(I\bigcup\{i\})\setminus\{j\}}\partial_j (m_{A,I}v_i)
$$
(let $I'=I\bigcup \{i\}$)
$$
=\sum_{A,|I'|\geq 2,i,j,}\mathbbm{1}_{\{i,j\in I'\}}(-1)^{O_{I'}(i)+O_{I'}(j)}m_{A,I'\setminus\{j\}}\partial_j (m_{A,I'\setminus\{i\}}v_i)
$$
( and, since $m_{A,I'\setminus\{i\}}v_i=\frac{m_{A,I'\setminus\{i\}}}{\tau}(\omega_i+m_{\alpha i}d_{\alpha})$)
$$
=\sum_{A,|I'|\geq 2,i,j,}\mathbbm{1}_{\{i,j\in I'\}}(-1)^{O_{I'}(i)+O_{I'}(j)}m_{A,I'\setminus\{j\}}\partial_j\left(\frac{m_{A,I'\setminus\{i\}}\omega_i+
m_{A,I'\setminus\{i\}}m_{\alpha i}d_{\alpha}}{\tau}\right)
$$

Also, we have
$$
\Sigma_4+\Sigma_5=-\sum_{A,I,\alpha,i}\mathbbm{1}_{\{\alpha\in A,i\in I\}}(-1)^{O_A(\alpha)+O_I(i)}m_{A\setminus \{\alpha\},I\setminus\{i\} }\partial_i(m_{A,I}d_{\alpha})
$$
(let $A'=A\setminus \{\alpha\}$)
$$
=-\tau\partial_i(m_{\alpha i}d_{\alpha})-\sum_{A',|I|\geq 2,\alpha,i}\mathbbm{1}_{\{\alpha\notin A',i\in I\}}(-1)^{O_{A'}(\alpha)+O_{I}(i)}m_{A',I\setminus\{i\}}\partial_i(m_{A'\bigcup\{\alpha\},I}d_{\alpha})
$$
Since $$\mathbbm{1}_{\{\alpha\notin A'\}}m_{A'\bigcup\{\alpha\},I}d_{\alpha}
$$
$$=\frac{1}{\tau}\Big(\sum_{j\in I}(-1)^{O_{A'}(\alpha)+O_{I}(j)}m_{A',I\setminus\{j\}}m_{\alpha j}-\varphi_{A',I}^{\alpha}\Big)d_{\alpha}$$
we have
$$
\Sigma_4+\Sigma_5=-\tau\partial_i(m_{\alpha i}d_{\alpha}) + \sum_{A',|I|\geq 2,\alpha,i}\mathbbm{1}_{\{i\in I\}}(-1)^{O_{A'}(\alpha)+O_{I}(i)}m_{A',I\setminus\{i\}}\partial_i\left(\frac{
\varphi_{A',I}^{\alpha}d_{\alpha}}{\tau}\right)
$$
$$
-\sum_{A',|I|\geq 2,\alpha,i,j}\mathbbm{1}_{\{i,j\in I\}}(-1)^{O_I(i)+O_I(j)}m_{A',I\setminus\{i\}}\partial_i\left(\frac{
m_{A',I\setminus\{j\}}m_{\alpha j}d_{\alpha}}{\tau}\right)
$$
We find that the last term is cancel when we add it up with $\Sigma_3+\Sigma_6$, so we have
$$
\partial_t\lambda=-v_j\partial_j\lambda + \tau\partial_i\omega_i +\sum_{A,|I'|\geq 2,i,j,}\mathbbm{1}_{\{i,j\in I'\}}(-1)^{O_{I'}(i)+O_{I'}(j)}m_{A,I'\setminus\{j\}}\partial_j\left(\frac{m_{A,I'\setminus\{i\}}\omega_i}{\tau}\right)
$$
$$
+ \sum_{A',|I|\geq 2,\alpha,i}\mathbbm{1}_{\{i\in I\}}(-1)^{O_{A'}(\alpha)+O_{I}(i)}m_{A',I\setminus\{i\}}\partial_i\left(\frac{
\varphi_{A',I}^{\alpha}d_{\alpha}}{\tau}\right)
$$


\subsection*{Equation for $\varphi^{\alpha}_{A,I}$}

Now let's find the equation for $\varphi^{\alpha}_{A,I}$. We only consider the case $|A|\geq 2$. We have
$$
\partial_t \varphi^{\alpha}_{A,I}=\partial_t\left(\sum_{i\in I}(-1)^{O_A(\alpha)+O_I(i)}m_{A,I\setminus\{i\}}m_{\alpha i} - \mathbbm{1}_{\{\alpha\notin A\}}\tau m_{A\bigcup\{\alpha\},I}\right)
$$
$$
=\sum_{i\in I}(-1)^{O_A(\alpha)+O_I(i)}m_{\alpha i}\big(m_{A,I\setminus\{i\}}\partial_jv_j-v_j\partial_j m_{A,I\setminus\{i\}}\big) +\Sigma_7+\Sigma_8
$$
$$
-\sum_{i\in I}(-1)^{O_A(\alpha)+O_I(i)}m_{A,I\setminus\{i\}}\big(v_j\partial_j m_{\alpha i} + m_{\alpha j}\partial_i v_j-m_{\alpha i}\partial_j v_j +\tau\partial_i d_{\alpha}\big)
$$
$$
-\mathbbm{1}_{\{\alpha\notin A\}}m_{A\bigcup\{\alpha\},I}(\tau\partial_jv_j-v_j\partial_j\tau)-\mathbbm{1}_{\{\alpha\notin A\}}\tau\big(m_{A\bigcup\{\alpha\},I}\partial_jv_j-v_j\partial_j m_{A\bigcup\{\alpha\},I}\big) +\Sigma_9+\Sigma_{10}
$$
$$
=2\varphi^{\alpha}_{A,I}\partial_j v_j-v_j\partial_j\varphi^{\alpha}_{A,I}-\sum_{i\in I}(-1)^{O_A(\alpha)+O_I(i)}m_{A,I\setminus\{i\}}\big(m_{\alpha j}\partial_i v_j +\tau\partial_i d_{\alpha}\big) +\Sigma_7+\Sigma_8+\Sigma_9+\Sigma_{10}
$$
where
$$
\Sigma_7=-\sum_{i,j,k}\mathbbm{1}_{\{i\neq j\in I,k\notin I\setminus\{i,j\}\}}(-1)^{O_{A}(\alpha)+O_I(i)+O_{I\setminus\{i\}}(j)+O_{I\setminus\{i,j\}}(k)} m_{A,(I\setminus\{i,j\})\bigcup\{k\}}m_{\alpha i}\partial_{j}v_k
$$
$$
\Sigma_8=-\sum_{\beta,i,j}\mathbbm{1}_{\{\beta\in A,i\neq j\in I\}}(-1)^{O_A(\alpha)+O_I(i)+O_A(\beta)+O_{I\setminus\{i\}}(j)} m_{A\setminus\{\beta\},I\setminus\{i,j\}} m_{\alpha i}\partial_{j}d_{\beta}
$$
$$
\Sigma_9=\sum_{j,k}\mathbbm{1}_{\{\alpha\notin A,j\in I,k\notin I\setminus\{j\}\}}(-1)^{O_{I\setminus\{j\}}(k)+O_{I}(j)} \tau m_{A\bigcup\{\alpha\},(I\setminus\{j\})\bigcup\{k\}}\partial_{j}v_k
$$
$$
\Sigma_{10}=\sum_{\beta,j}\mathbbm{1}_{\{\alpha\notin A,\beta\in A\bigcup\{\alpha\},j\in I\}}(-1)^{O_{A\bigcup\{\alpha\}}(\beta)+O_I(j)}\tau m_{(A\bigcup\{\alpha\})\setminus\{\beta\},I\setminus\{j\}} \partial_{j}d_{\beta}
$$
Now let's look at $\Sigma_7$, first, since
$$\mathbbm{1}_{\{i\neq j\in I,k\notin I\setminus\{i,j\}\}}=\mathbbm{1}_{\{j\in I,i\in I\setminus \{j\},k\notin I\setminus\{j\}\}}+\mathbbm{1}_{\{j\in I,i\in I\setminus \{j\},k=i\}}$$
we have
$$
\Sigma_7=-\sum_{i,j,k}\mathbbm{1}_{\{j\in I,i\in I\setminus \{j\},k\notin I\setminus\{j\}\}}(-1)^{O_{A}(\alpha)+O_I(i)+O_{I\setminus\{i\}}(j)+O_{I\setminus\{i,j\}}(k)} m_{A,(I\setminus\{i,j\})\bigcup\{k\}}m_{\alpha i}\partial_{j}v_k
$$
$$
+\sum_{j,k}\mathbbm{1}_{\{j\in I,k\in I\setminus \{j\}\}}(-1)^{O_A(\alpha)+O_I(j)}m_{A,I\setminus\{j\}}m_{\alpha k}\partial_j v_k
$$
Here we use the fact that, for $k\neq j$, $$O_I(k)+O_{I\setminus\{k\}}(j)\equiv O_I(j)+O_{I\setminus\{j\}}(k) +1\;\;({\rm mod\ }2).$$
Further more, for different number $i,j,k$, we have the following equality
$$O_I(i)+O_{I\setminus\{i\}}(j)+O_{I\setminus\{i,j\}}(k)\equiv O_{I\setminus\{j\}}(k)+O_{I}(j)+O_{(I\setminus\{j\})\bigcup\{k\}}(i)\;\;({\rm mod\ }2)$$
Then,
we have
$$
\Sigma_7=-\sum_{i,j,k}\mathbbm{1}_{\{j\in I,i\in I\setminus \{j\},k\notin I\setminus\{j\}\}}(-1)^{O_{A}(\alpha)+O_{I\setminus\{j\}}(k)+O_{I}(j)+O_{(I\setminus\{j\})\bigcup\{k\}}(i)} m_{A,(I\setminus\{i,j\})\bigcup\{k\}}m_{\alpha i}\partial_{j}v_k
$$
$$
+\sum_{j\in I,k}(-1)^{O_A(\alpha)+O_I(j)}m_{A,I\setminus\{j\}}m_{\alpha k}\partial_j v_k - \sum_{j,k}\mathbbm{1}_{\{j\in I,k\notin I\setminus \{j\}\}}(-1)^{O_A(\alpha)+O_I(j)}m_{A,I\setminus\{j\}}m_{\alpha k}\partial_j v_k
$$
(since $\mathbbm{1}_{\{j\in I,i\in I\setminus \{j\},k\notin I\setminus\{j\}\}} + \mathbbm{1}_{\{j\in I,k\notin I\setminus \{j\},i=k\}}=\mathbbm{1}_{\{j\in I,k\notin I\setminus \{j\},i\in (I\setminus \{j\})\bigcup\{k\}\}}$)
$$
=-\sum_{i,j,k}\mathbbm{1}_{\{j\in I,k\notin I\setminus\{j\},i\in (I\setminus \{j\})\bigcup\{k\}\}}(-1)^{O_{A}(\alpha)+O_{I\setminus\{j\}}(k)+O_{I}(j)+O_{(I\setminus\{j\})\bigcup\{k\}}(i)} m_{A,((I\setminus\{j\})\bigcup\{k\})\setminus\{i\}}m_{\alpha i}\partial_{j}v_k
$$
$$
+\sum_{j\in I,k}(-1)^{O_A(\alpha)+O_I(j)}m_{A,I\setminus\{j\}}m_{\alpha k}\partial_j v_k
$$
Together with $\Sigma_9$, we have
$$
\Sigma_7 +\Sigma_9 -\sum_{j\in I,k}(-1)^{O_A(\alpha)+O_I(j)}m_{A,I\setminus\{j\}}m_{\alpha k}\partial_j v_k
=-\sum_{j,k}\mathbbm{1}_{\{j\in I,k\notin I\setminus\{j\}\}}(-1)^{O_{I\setminus\{j\}}(k)+O_{I}(j)} \varphi_{A,(I\setminus\{j\})\bigcup\{k\}}^{\alpha}\partial_{j}v_k
$$
Now let's look at $\Sigma_{10}$. Since
$$\mathbbm{1}_{\{\alpha\notin A,\beta\in A\bigcup\{\alpha\},j\in I\}}=\mathbbm{1}_{\{\alpha\notin A,\beta\in A,j\in I\}}+\mathbbm{1}_{\{\alpha\notin A,\beta=\alpha,j\in I\}}$$
and, for any $\alpha\notin A$, $\beta\in A$,
$$
O_{A\setminus\{\beta\}}(\alpha)+O_{A\bigcup\{\alpha\}}(\beta)\equiv O_{A}(\alpha)+O_{A}(\beta) +1\;\;({\rm mod\ }2)
$$
then we have
$$
\Sigma_{10}=-\sum_{\beta,j}\mathbbm{1}_{\{\alpha\notin A,\beta\in A,j\in I\}}(-1)^{O_{A}(\alpha)+O_{A}(\beta)+O_{A\setminus\{\beta\}}(\alpha)+O_I(j)}\tau m_{(A\setminus\{\beta\})\bigcup\{\alpha\},I\setminus\{j\}} \partial_{j}d_{\beta}
$$
$$
+\sum_{j\in I}\mathbbm{1}_{\{\alpha\notin A\}}(-1)^{O_{A}(\alpha)+O_I(j)}\tau m_{A,I\setminus\{j\}} \partial_{j}d_{\alpha}
$$
Now because $\mathbbm{1}_{\{\alpha\notin A,\beta\in A,j\in I\}}=\mathbbm{1}_{\{\alpha\notin A\setminus\{\beta\},\beta\in A,j\in I\}}-\mathbbm{1}_{\{\alpha=\beta\in A,j\in I\}}$, we have
$$
\Sigma_{10}=-\sum_{\beta,j}\mathbbm{1}_{\{\alpha\notin A\setminus\{\beta\},\beta\in A,j\in I\}}(-1)^{O_{A}(\alpha)+O_{A}(\beta)+O_{A\setminus\{\beta\}}(\alpha)+O_I(j)}\tau m_{(A\setminus\{\beta\})\bigcup\{\alpha\},I\setminus\{j\}} \partial_{j}d_{\beta}
$$
$$
+\sum_{j\in I}\big(\mathbbm{1}_{\{\alpha\in A\}}+\mathbbm{1}_{\{\alpha\notin A\}}\big)(-1)^{O_{A}(\alpha)+O_I(j)}\tau m_{A,I\setminus\{j\}} \partial_{j}d_{\alpha}
$$
Since for any $i\neq j\in I$,
$$O_I(i)+O_{I\setminus\{i\}}(j)\equiv O_I(j)+O_{I\setminus\{j\}}(i) +1\;\;({\rm mod\ }2)$$
Then we have
$$
\Sigma_8+\Sigma_{10} -\sum_{j\in I}(-1)^{O_{A}(\alpha)+O_I(j)}\tau m_{A,I\setminus\{j\}} \partial_{j}d_{\alpha}
$$
$$
=\sum_{\beta,j}\mathbbm{1}_{\{\beta\in A,j\in I\}}(-1)^{O_{A}(\alpha)+O_{A}(\beta)+O_{A\setminus\{\beta\}}(\alpha)+O_I(j)}\varphi_{A\setminus\{\beta\},I\setminus\{j\}}^{\alpha} \partial_{j}d_{\beta}
$$
In summary, we have
$$
\partial_t \varphi^{\alpha}_{A,I} =2\varphi^{\alpha}_{A,I}\partial_j v_j-v_j\partial_j\varphi^{\alpha}_{A,I} +\sum_{\beta,j}\mathbbm{1}_{\{\beta\in A,j\in I\}}(-1)^{O_{A}(\alpha)+O_{A}(\beta)+O_{A\setminus\{\beta\}}(\alpha)+O_I(j)}\varphi_{A\setminus\{\beta\},I\setminus\{j\}}^{\alpha} \partial_{j}d_{\beta}
$$
$$
-\sum_{j,k}\mathbbm{1}_{\{j\in I,k\notin I\setminus\{j\}\}}(-1)^{O_{I\setminus\{j\}}(k)+O_{I}(j)} \varphi_{A,(I\setminus\{j\})\bigcup\{k\}}^{\alpha}\partial_{j}v_k
$$


\subsection*{Equation for $\psi_{A,I}^{i}$}

Now we compute $\psi_{A,I}^{i}$, and we only consider the case when $|I|\geq 2$. We have
$$
\partial_t\psi_{A,I}^{i}=\partial_t\Big(\sum_{\alpha\in A}(-1)^{O_A(\alpha)+O_I(i)}m_{A\setminus\{\alpha\},I}m_{\alpha i} - \mathbbm{1}_{\{i\notin I\}}\tau m_{A,I\bigcup\{i\}}\Big)
$$
$$
=\sum_{\alpha\in A}(-1)^{O_A(\alpha)+O_I(i)}m_{\alpha i}\big(m_{A\setminus\{\alpha\},I}\partial_jv_j-v_j\partial_j m_{A\setminus\{\alpha\},I}\big) +\Sigma_{11}+\Sigma_{12}
$$
$$
-\sum_{\alpha\in A}(-1)^{O_A(\alpha)+O_I(i)}m_{A\setminus\{\alpha\},I}\big(v_j\partial_j m_{\alpha i} + m_{\alpha j}\partial_i v_j-m_{\alpha i}\partial_j v_j +\tau\partial_i d_{\alpha}\big)
$$
$$
-\mathbbm{1}_{\{i\notin I\}}\tau\big(m_{A,I\bigcup\{i\}}\partial_jv_j-v_j\partial_j m_{A,I\bigcup\{i\}}\big) +\Sigma_{13}+\Sigma_{14}
$$
$$
-\mathbbm{1}_{\{i\notin I\}}m_{A,I\bigcup\{i\}}(\tau\partial_jv_j-v_j\partial_j\tau)
$$
$$
=2\psi_{A,I}^{i}\partial_j v_j-v_j\partial_j\psi_{A,I}^{i} +\Sigma_{11}+\Sigma_{12}+\Sigma_{13}+\Sigma_{14}
$$
$$
-\sum_{\alpha\in A}(-1)^{O_A(\alpha)+O_I(i)}m_{A\setminus\{\alpha\},I}\big(m_{\alpha j}\partial_i v_j +\tau\partial_i d_{\alpha}\big)
$$
where
$$
\Sigma_{11}=-\sum_{\alpha,j,k}\mathbbm{1}_{\{\alpha\in A,j\in I,k\notin I\setminus\{j\}\}}(-1)^{O_{A}(\alpha)+O_I(i)+O_{I}(j)+O_{I\setminus\{j\}}(k)} m_{A\setminus\{\alpha\},(I\setminus\{j\})\bigcup\{k\}}m_{\alpha i}\partial_{j}v_k
$$
$$
\Sigma_{12}=-\sum_{\alpha,\beta,j}\mathbbm{1}_{\{\alpha\neq\beta\in A,j\in I\}}(-1)^{O_A(\alpha)+O_I(i)+O_{A\setminus \{\alpha\}}(\beta)+O_{I}(j)} m_{A\setminus\{\alpha,\beta\},I\setminus\{j\}} m_{\alpha i}\partial_{j}d_{\beta}
$$
$$
\Sigma_{13}=\sum_{j,k}\mathbbm{1}_{\{i\notin I,j\in I\bigcup\{i\},k\notin (I\bigcup\{i\})\setminus\{j\}\}}(-1)^{O_{(I\bigcup\{i\})\setminus\{j\}}(k)+O_{I\bigcup\{i\}}(j)} \tau m_{A,((I\bigcup\{i\})\setminus\{j\})\bigcup\{k\}}\partial_{j}v_k
$$
$$
\Sigma_{14}=\sum_{\beta,j}\mathbbm{1}_{\{\beta\in A,i\notin I,j\in I\bigcup\{i\}\}}(-1)^{O_A(\beta)+O_{I\bigcup\{i\}}(j)}\tau m_{A\setminus\{\beta\},(I\bigcup\{i\})\setminus\{j\}} \partial_{j}d_{\beta}
$$
First, let us consider the most complicated part $\Sigma_{13}$.
Since
$$\mathbbm{1}_{\{i\notin I,j\in I\bigcup\{i\},k\notin (I\bigcup\{i\})\setminus\{j\}\}}=\mathbbm{1}_{\{i\notin I,j\in I,k\notin (I\setminus\{j\})\bigcup\{i\}\}} + \mathbbm{1}_{\{i\notin I,j=i,k\notin I\}}$$
we have
$$\Sigma_{13}=\sum_{k}\mathbbm{1}_{\{i\notin I\}}(-1)^{O_I(k)+O_I(i)}\tau m_{A,I\bigcup\{k\}} + \Sigma'_{13}$$
where
$$
\Sigma'_{13}=\sum_{j,k}\mathbbm{1}_{\{i\notin I,j\in I,k\notin (I\setminus\{j\})\bigcup\{i\}\}}(-1)^{O_{(I\setminus\{j\})\bigcup\{i\}}(k)+O_{I\bigcup\{i\}}(j)} \tau m_{A,((I\setminus\{j\})\bigcup\{i\})\bigcup\{k\}}\partial_{j}v_k
$$
Now for $i\notin I$, $j\in I$, $k\notin I\setminus\{j\}$, $k\neq i$, we can prove that
$$O_{(I\setminus\{j\})\bigcup\{i\}}(k)+O_{I\bigcup\{i\}}(j) \equiv O_I(i)+O_{I}(j)+O_{I\setminus\{j\}}(k)+O_{(I\setminus\{j\})\bigcup\{k\}}(i)\;\;({\rm mod\ }2)$$
(We can show this equality by discussing in the 4 cases: $i<\min\{j,k\}$, $i>\max\{j,k\}$, $j<i<k$ and $k<i<j$.) So $\Sigma'_{13}$ can be written as
$$
\Sigma'_{13}=\sum_{j,k}\mathbbm{1}_{\{i\notin I,j\in I,k\notin (I\setminus\{j\})\bigcup\{i\}\}}(-1)^{S} \tau m_{A,((I\setminus\{j\})\bigcup\{i\})\bigcup\{k\}}\partial_{j}v_k
$$
where $S=O_I(i)+O_{I}(j)+O_{I\setminus\{j\}}(k)+O_{(I\setminus\{j\})\bigcup\{k\}}(i)$. Now we claim that
$$\mathbbm{1}_{\{i\notin I,j\in I,k\notin (I\setminus\{j\})\bigcup\{i\}\}}=\mathbbm{1}_{\{i\notin (I\setminus\{j\})\bigcup\{k\},j\in I,k\notin I\setminus\{j\}\}}-\mathbbm{1}_{\{i=j,j\in I,k\notin I\}}$$
(In fact, we can prove the equality step by step:
$$\mathbbm{1}_{\{i\notin I,j\in I,k\notin (I\setminus\{j\})\bigcup\{i\}\}} = \mathbbm{1}_{\{i\notin I,j\in I,k\notin I\setminus\{j\}\}} -\mathbbm{1}_{\{i\notin I,j\in I,k=i\}} $$
$$\mathbbm{1}_{\{i\notin I,j\in I,k\notin I\setminus\{j\}\}} =
\mathbbm{1}_{\{i\notin I\setminus\{j\},j\in I,k\notin I\setminus\{j\}\}} - \mathbbm{1}_{\{i=j,j\in I,k\notin I\setminus\{j\}\}}$$
$$\mathbbm{1}_{\{i\notin I\setminus\{j\},j\in I,k\notin I\setminus\{j\}\}}=\mathbbm{1}_{\{i\notin (I\setminus\{j\})\bigcup\{k\},j\in I,k\notin I\setminus\{j\}\}} + \mathbbm{1}_{\{i=k,j\in I,k\notin I\setminus\{j\}\}}$$
$$\mathbbm{1}_{\{i=j,j\in I,k\notin I\setminus\{j\}\}}=\mathbbm{1}_{\{i=j,j\in I,k\notin I\}}+\mathbbm{1}_{\{i=j=k,j\in I\}}$$
$$\mathbbm{1}_{\{i=k,j\in I,k\notin I\setminus\{j\}\}}=\mathbbm{1}_{\{i=k,j\in I,k\notin I\}}+\mathbbm{1}_{\{i=j=k,j\in I\}}$$
by adding up the above equalities we get the desired result.)
Then $\Sigma'_{13}$ can be written as
$$
\Sigma'_{13}=\sum_{j,k}\mathbbm{1}_{\{i\notin (I\setminus\{j\})\bigcup\{k\},j\in I,k\notin I\setminus\{j\}\}}(-1)^{S} \tau m_{A,((I\setminus\{j\})\bigcup\{k\})\bigcup\{i\}}\partial_{j}v_k
$$
$$
-\sum_{k}\mathbbm{1}_{\{i\in I\}}(-1)^{O_{I\setminus\{i\}}(k)+O_{I\bigcup\{k\}}(i)}\tau m_{A,I\bigcup\{k\}}
$$
We can easily verify that for $i\in I$, $k\notin I$,
$$O_{I\setminus\{i\}}(k)+O_{I\bigcup\{k\}}(i) \equiv O_{I}(k)+O_{I}(i)+1\;\;({\rm mod\ }2)$$
So we have
$$\Sigma_{13}=\sum_{k}(-1)^{O_I(k)+O_I(i)}\tau m_{A,I\bigcup\{k\}}$$
$$+\sum_{j,k}\mathbbm{1}_{\{i\notin (I\setminus\{j\})\bigcup\{k\},j\in I,k\notin I\setminus\{j\}\}}(-1)^{S} \tau m_{A,((I\setminus\{j\})\bigcup\{k\})\bigcup\{i\}}\partial_{j}v_k$$
Then we have
$$
\Sigma_{11}+\Sigma_{13}-\sum_{\alpha\in A,k}(-1)^{O_A(\alpha)+O_I(i)}m_{A\setminus\{\alpha\},I}m_{\alpha k}\partial_i v_k
$$
$$
=-\sum_{k}(-1)^{O_I(i)+O_I(k)}\psi_{A,I}^{k}\partial_iv_k-\sum_{j,k}\mathbbm{1}_{\{j\in I,k\notin I\setminus\{j\}\}}(-1)^{S} \psi_{A,(I\setminus\{j\})\bigcup\{k\}}^{i}\partial_{j}v_k
$$
Now let's look at $\Sigma_{14}$. Since
$$
\mathbbm{1}_{\{\beta\in A,i\notin I,j\in I\bigcup\{i\}\}}=\mathbbm{1}_{\{\beta\in A,i\notin I,j\in I\}}+\mathbbm{1}_{\{\beta\in A,i\notin I,j=i\}}
$$
we have
$$
\Sigma_{14}=\sum_{\beta\in A}\mathbbm{1}_{\{i\notin I\}}(-1)^{O_A(\beta)+O_{I}(i)}\tau m_{A\setminus\{\beta\},I} \partial_{i}d_{\beta} + \Sigma'_{14}
$$
where
$$
\Sigma'_{14}=\sum_{\beta,j}\mathbbm{1}_{\{\beta\in A,i\notin I,j\in I\}}(-1)^{O_A(\beta)+O_{I\bigcup\{i\}}(j)}\tau m_{A\setminus\{\beta\},(I\bigcup\{i\})\setminus\{j\}} \partial_{j}d_{\beta}
$$
We can check that for $i\notin I$, $j\in I$, we have
$$O_I(i)+O_{I\setminus\{j\}}(i) + O_I(j)+O_{I\bigcup\{i\}}(j) \equiv 1\;\;({\rm mod\ }2)$$
So $\Sigma'_{14}$ can be written as
$$
\Sigma'_{14}=-\sum_{\beta,j}\mathbbm{1}_{\{\beta\in A,i\notin I,j\in I\}}(-1)^{O_A(\beta)+O_I(i)+O_{I\setminus\{j\}}(i) + O_I(j)}\tau m_{A\setminus\{\beta\},(I\bigcup\{i\})\setminus\{j\}} \partial_{j}d_{\beta}
$$
Because
$$\mathbbm{1}_{\{\beta\in A,i\notin I,j\in I\}}=\mathbbm{1}_{\{\beta\in A,i\notin I\setminus\{j\},j\in I\}}-\mathbbm{1}_{\{\beta\in A,i=j,j\in I\}}$$
Then
$$
\Sigma'_{14}=-\sum_{\beta,j}\mathbbm{1}_{\{\beta\in A,i\notin I\setminus\{j\},j\in I\}}(-1)^{O_A(\beta)+O_I(i)+O_{I\setminus\{j\}}(i) + O_I(j)}\tau m_{A\setminus\{\beta\},(I\bigcup\{i\})\setminus\{j\}} \partial_{j}d_{\beta}
$$
$$
+\sum_{\beta\in A}\mathbbm{1}_{\{i\in I\}}(-1)^{O_A(\beta)+O_{I}(i)}\tau m_{A\setminus\{\beta\},I} \partial_{i}d_{\beta}
$$
Since for $\alpha\neq\beta\in A$,
$$O_A(\alpha)+O_{A\setminus\{\alpha\}}(\beta) \equiv  O_A(\beta)+O_{A\setminus\{\beta\}}(\alpha) + 1\;\;({\rm mod\ }2)$$
We finally have
$$
\Sigma_{12}+\Sigma_{14}=\sum_{\beta\in A,j\in I}(-1)^{O_A(\beta) + O_I(j) +O_I(i)+O_{I\setminus\{j\}}(i)}\psi_{A\setminus\{\beta\},I\setminus\{j\}}^{i}
\partial_jd_{\beta}
$$
$$
+\sum_{\beta\in A}(-1)^{O_A(\beta)+O_{I}(i)}\tau m_{A\setminus\{\beta\},I} \partial_{i}d_{\beta}
$$
In summary, we have
$$
\partial_t\psi_{A,I}^{i}=2\psi_{A,I}^{i}\partial_j v_j-v_j\partial_j\psi_{A,I}^{i} -\sum_{k}(-1)^{O_I(i)+O_I(k)}\psi_{A,I}^{k}\partial_iv_k
$$
$$
+ \sum_{\beta\in A,j\in I}(-1)^{O_A(\beta) + O_I(j) +O_I(i)+O_{I\setminus\{j\}}(i)}\psi_{A\setminus\{\beta\},I\setminus\{j\}}^{i}\partial_jd_{\beta}
$$
$$
-\sum_{j,k}\mathbbm{1}_{\{j\in I,k\notin I\setminus\{j\}\}}(-1)^{O_I(i)+O_{I}(j)+O_{I\setminus\{j\}}(k)+O_{(I\setminus\{j\})\bigcup\{k\}}(i)} \psi_{A,(I\setminus\{j\})\bigcup\{k\}}^{i}\partial_{j}v_k
$$


\begin{thebibliography}{99}
\bibitem{Br} Y. Brenier,
{\it Hydrodynamic structure of the augmented Born-Infeld equations,}
{Arch. Ration. Mech. Anal. 172 (2004) 65-91.}
\bibitem{BD} Y. Brenier, X. Duan,
{\it From conservative to dissipative systems through quadratic change of time, with
application to the curve-shortening flow,}
{preprint 2017.}
\bibitem{BY} Y. Brenier, W.-A. Yong
{\it Derivation of particle, string and membrane motions from the Born-Infeld electromagnetism,}
{J. Math. Phys. 46 (2005), no. 6, 062305, 17 pp.}
%
%
\bibitem{Da} C. M. Dafermos,
{\it Hyperbolic conservation laws in continuum physics,}
{Springer, Berlin, 2000.}
\bibitem{DST} S. Demoulini, D. Stuart, A. Tzavaras,
{\sl Weak-strong uniqueness of dissipative measure-valued solutions for polyconvex elastodynamics,}
{Arch. Ration. Mech. Anal. 205 (2012) 927-961.}
\bibitem{Lind} H. Lindblad,
{\it A remark on global existence for small initial data of the minimal surface equation in Minkowskian space time,}
{Proc. Amer. Math. Soc. 132 (2004), pp. 1095-1102.}
\bibitem{Qin} T. Qin,
{\it Symmetrizing nonlinear elastodynamic system,}
{J. Elasticity 50 (1998), no. 3, 245-252.}
\bibitem{Po} J. Polchinski,
{\it String theory. Vol. I.}
{Cambridge University Press, 1998.}
\bibitem{Se} D. Serre,
{\it Hyperbolicity of the nonlinear models of Maxwell's equations},
{Arch. Ration. Mech. Anal. 172 (2004) 309-331.}
\end{thebibliography}
\end{document}